\documentclass[11pt,a4paper]{article}

\usepackage[english]{babel}
\usepackage[utf8]{inputenc}
\usepackage[OT1]{fontenc}
\usepackage{titlesec}
\usepackage{csquotes}

\usepackage[backend=biber, maxnames=50, style=alphabetic]{biblatex}
\usepackage[right=3cm, a4paper]{geometry}

\usepackage{amsmath}
\usepackage{amssymb}	

\usepackage{comment}
\usepackage{mathtools}
\usepackage{tikz-cd}
\usepackage{float}
\usepackage{graphicx}
\usepackage{xcolor}
\usepackage{quiver}

\usepackage[colorlinks=true, allcolors=blue]{hyperref}
\usepackage{enumitem}
\usepackage{svg}
\usepackage{scalerel}
\usepackage{overpic}
\usepackage{multicol}
\usepackage{framed}
\usepackage[font=small, labelfont=it, textfont=it]{caption}
\usepackage[font=small]{subcaption}

\usepackage{pdfpages}

\usepackage[dvipsnames]{xcolor}

\newcommand{\tr}{\operatorname{tr}}

\newcommand{\g}{\ensuremath{\mathfrak{g}}}
\newcommand{\h}{\ensuremath{\mathfrak{h}}}

\newcommand{\Ca}{\ensuremath{\mathcal{C}}}

\newcommand{\Ad}{\ensuremath{\operatorname{Ad}}}
\newcommand{\Sym}{\ensuremath{\operatorname{Sym}}}
\newcommand{\Flip}{\ensuremath{\operatorname{Flip}}}

\newcommand{\Hom}{\ensuremath{\operatorname{Hom}}}

\newcommand{\Umod}[2][\Phi]{U(#2)\text{-}\mathrm{mod}_\hbar^{#1}} 
\newcommand{\Umodnp}[2][\Phi]{U{#2}\text{-}\mathrm{mod}_\hbar^{#1}} 

\usepackage{amsthm}
\newcounter{dummy} \numberwithin{dummy}{section}
\newtheorem{theorem}[dummy]{Theorem}
\newtheorem{proposition}[dummy]{Proposition}
\newtheorem{corollary}[dummy]{Corollary}

\theoremstyle{definition}
\newtheorem{definition}[dummy]{Definition} 
\newtheorem{remark}[dummy]{Remark} 
\newtheorem{example}[dummy]{Example}

\begin{document}

\title{Pentagon Equation in Deformation Quantization of Character Varieties}
\author{Ján Pulmann \\ Mathematical Institute, Charles University of Prague \\  Institut für Mathematik, Humboldt-Universität zu Berlin}
\date{}

\maketitle

\begin{abstract}
Deformation quantizations of the moduli space of framed flat bundles live naturally in the Drinfeld category. Such quantizations were constructed using the so-called quantum fusion by Li-Bland and Ševera. We give a complete characterization of the dependence of this quantization on combinatorial choices, and prove that these quantizations are all isomorphic, forming vertices of an explicitly constructed complete commuting graph of algebra isomorphism. To this end, we extend the Kontsevich integral methods of Le and Murakami from even associators to all associators. Finally, we also prove that bi-invariant Haar measures provide traces on these quantizations.  
\end{abstract}

\tableofcontents

\section{Introduction}
The phase space of Chern--Simons theory, i.e.\ the moduli space of flat connections on a surface, has a natural Poisson structure, due to Atiyah--Bott \cite{AtiyahBott} and Goldman \cite{Goldman86}. This Poisson structure has been quantized by many authors \cite{AGS95, Andersen1998, RocheSzenes02, LBS:DQ, BZBJ:integrating}. In this work, we study the deformation quantization of Li-Bland and Ševera \cite{LBS:DQ, LBS:coherent}, which goes through quasi-Poisson geometry of framed moduli spaces \cite{AKSM, LBS:poisson}. This quantization is defined in a straightforward way, gluing the moduli space from simple building blocks (framed moduli spaces on a disk) using a quantum version of the quasi-Poisson fusion. As a consequence, the quantization depends on a decomposition of the surface into disks.

We completely characterize the dependence of this quantization on this decomposition.  We get one star product $\star_\Gamma$ for each decomposition of the surface and a \emph{explicit} isomorphism  $\Psi_{\Gamma \to \Gamma'}$ between these star products such that the following diagrams commute
\begin{equation} \label{diag:coherent-quantizations}\begin{tikzcd}
	{\star_\Gamma} && {\star_{\Gamma''}} \\
	& {\star_{\Gamma'}}
	\arrow["{\Psi_{\Gamma \to \Gamma''}}", from=1-1, to=1-3]
	\arrow["{\Psi_{\Gamma \to \Gamma'}}"', from=1-1, to=2-2]
	\arrow["{\Psi_{\Gamma' \to \Gamma''}}"', from=2-2, to=1-3]
\end{tikzcd}.\end{equation}
The isomorphisms $\Psi$ are constructed as follows. The \emph{flip}, the elementary move between surface decompositions, is implemented by the action of the Drinfeld associator. In the \emph{Ptolemy} groupoid of decompositions, where arrows are generated by flips, the relations are given by a pentagon equation. Drinfeld associators satisfy a corresponding pentagon equation. Therefore, we conclude that the star products form an honest representation of the Ptolemy groupoid in the category of associative algebras in Drinfeld's braided monoidal category $\Umodnp{\g}$ \cite{DrinfeldGal, Cartier}. This provides an explicit connection between these two pentagon equations. These results also apply to moduli spaces of unframed flat bundles, i.e. to character varieties, by taking $G$-invariants.

Independence of star products on triangulations is easy to conclude when using the formalism of factorization homology to construct these quantizations. Indeed, in my work with Karlsson, Keller and M\"uller \cite{KKMP}, we relate the quantization of Li-Bland and Ševera to the deformation quantization coming from factorization homology of \cite{BZBJ:integrating}, giving an abstract proof of the triangulation independence. Thus, the main result of the present paper can be seen as an alternative, explicit proof of this result \cite{KKMP}, giving explicit isomorphisms using the Kontsevich integral. 
\smallskip

To give such explicit isomorphisms, we extend to general associators the Kontsevich integral techniques developed by Le and Murakami in \cite{LeMurakami:universal, LeMurakami:parallel} for even associators. An important technical step is relating the operation which multiplies each chord with $-1$ to the reflection around a horizontal axis: this allows us to simplify the diagrams appearing in \cite{LBS:DQ} quantization. Such operations appeared in the non-perturbative setting in $K_{\mathfrak{q}}$ algebras of \cite{Gaiotto:Kq-algebras}.

We believe that the methods we use here, namely Drinfeld's quantization of infinitesimally braided monoidal categories \cite{DrinfeldGal, Cartier, BarNatanOnAssociators} and the Kontsevich integral \cite{KontsevichVassiliev, BarNatan:vassiliev, LeMurakami:parallel, LeMurakami:universal}, are highly effective \emph{computational} tools for quantizations of moduli spaces of flat connections. We illustrate this by giving an explicit twist relating the associators $\Phi(X, Y)$ and $\Phi(-X, -Y)$ and constructing  traces on \cite{LBS:DQ} quantizations. It would be interesting to interpret these results back in terms of braided monoidal categories and factorization homology.

\subsection{Acknowledgement}
This work is an extension of an unpublished part of my Ph.D. thesis \cite{Jan:thesis}, where the main result appeared. This article is written as an independent text; my thesis is referenced in a couple of places for additional results not needed for the main argument. 

I would like to thank my Ph.D. advisor Pavol Ševera for his guidance and many discussions which led to this work.  I would like to thank Eilind Karlsson, Corina Keller and Lukas Müller, as our joint work shaped the current paper. Additionally, I would like to thank Anton Alekseev, Dror Bar-Natan, Adrien Brochier, Thang Le, Florian Naef, Nikolai Perry, Pavel Safronov and Andras Szenes for numerous discussions at various stages of this project. I would like to acknowledge the support of the GAČR grant PIF-25-17640I.

\paragraph{AI use:} After writing the first draft, I had Claude Opus 5.5 generate a ``pre-referee report'' and implemented some of its suggestions. These amounted to mild mathematical corrections and typos. The report is available in the arxiv source code. In addition to this, I used Claude to find the references \cite{BarNatan:braids, Lusztig:QG, Gaitsgory:KZ}.

\section{Background}
\subsection{Quasi-Poisson geometry of moduli spaces of flat connections}

The relevant mathematical ingredients we need from Chern--Simons theory to construct the phase spaces are a connected Lie group $G$ with a \emph{Casimir} Lie algebra $\g$ and a surface $\Sigma$ with boundary. The ``phase space'' will be (a variant of) the moduli space of flat connections on $\Sigma$.

We can identify flat $G$-bundles on $\Sigma$ with equivalence classes of group homomorphisms 
\[  M_{\Sigma}(G) \coloneq \Hom_\text{Grp}(\pi_1(\Sigma), G)/G\,,\] where two morphisms are identified if they differ by conjugation via $g\in G$. These moduli spaces are usually not smooth, for example the moduli space of flat bundles over the annulus is the set of conjugacy classes of $G$. 

If $\Sigma$ is connected and has non-empty boundary, then the  \emph{framed} moduli space 
\[\Hom_\text{Grp}(\pi_1(\Sigma), G)\] is isomorphic to a product of $G$'s. An invariant bivector on this space, which descends to the Atiyah--Bott--Goldman Poisson bracket on $M_{\Sigma}(G)$, was constructed by Alekseev, Kosmann-Schwarzbach, Malkin and Meinrenken \cite{AMM, AKS, AKSM} and generalized by Li-Bland and Ševera in \cite{LBS:poisson, LBS:symplectic-poisson}. We will now recall this construction. 

The \emph{framed moduli spaces} we will study are defined as the space of isomorphism classes of principal $G$-bundles\footnote{Note that, since $G$ is connected and $\Sigma$ has boundary, all these principal bundles are trivial.} over $\Sigma$ equipped with a flat connection and a chosen trivialization of fibers over $v \in V\subset \partial \Sigma$, called framing. They are taken modulo bundle isomorphisms over the identity $\Sigma \to \Sigma$ which relate the flat connections and trivializations. Via Riemann-Hilbert correspondence, we can represent each such class in terms of holonomies as follows \cite[Sec.~6]{LBS:poisson}.
\begin{definition}
Let $(\Sigma, V)$ be a \emph{marked surface}, i.e.\ a surface $\Sigma$ and $V\subset \partial \Sigma$ a finite set meeting each component\footnote{but not necessarily each component of $\partial \Sigma$} of $\Sigma$. The fundamental groupoid $\Pi_1(\Sigma, V)$ is the groupoid with the set $V$ of objects and homotopy classes of paths on $\Sigma$ as morphisms. Let $G$ be a connected Lie group. The \emph{moduli space of flat principal $G$-bundles framed at $V$} is the space of groupoid morphisms
\begin{equation}
	M_{\Sigma, V}(G) \coloneqq \Hom_\text{Grpd}(\Pi_1(\Sigma, V), BG).
\end{equation} 
Here, $BG$ is the groupoid with one object and $G$ worth of morphisms.
\end{definition}
For $\Sigma$ connected with genus $g$ and $n$ boundary components, one can identify \[M_{\Sigma, V}(G)\cong G^{2g+n-2 + |V|}\] by choosing a generating set of edges for $\Pi_1(\Sigma, V)$, e.g.\ a skeleton. We review such parametrizations of the moduli space in the next section.
\smallskip

In order to define the bivector field on the moduli space, the Lie algebra $\g$ of $G$ should come with a Casimir element $t$ in the following sense.

\begin{definition}
A \emph{Casimir Lie algebra} is a Lie algebra $\g$ together with an invariant symmetric $2$-tensor $t \in (\Sym^2{\g})^\g$. The corresponding Cartan element $\phi \in (\wedge^3 \g)^\g$ is defined as\footnote{Going between polyvectors and elements of $\g^{\otimes{k}}$ requires choosing the normalization of the wedge product, or of the embedding $\wedge^k \g \to \g^{\otimes k}$. We use the unnormalized antisymmetrization.}    \begin{equation*}
    	\phi = -\tfrac {1}{24} [t_{1,2}, t_{2,3}] \in  \wedge^3 \g \subset (U\g)^{\otimes 3},
    \end{equation*}
where $t_{1,2} = t \otimes 1 \in (U\g)^{\otimes 3}$ and $t_{2,3} = 1 \otimes t \in (U\g)^{\otimes 3}$.

    If $(\g, t)$ is a fixed Casimir Lie algebra, we denote by $\bar\g$ the Casimir Lie algebra $(\g, -t)$. 
\end{definition}
Semisimple Lie algebras, and more generally Lie algebras equipped with a non-degenerate symmetric invariant pairing, give examples of Casimir Lie algebras. An abelian Lie algebra is Casimir for any symmetric $2$-tensor.

The bivector field on the framed moduli space of flat connections will not be Poisson, but instead quasi-Poisson in the following sense \cite{AKS, AKSM}.
\begin{definition}	
	Let $G$ be a Lie group with a Casimir Lie algebra. A $G$-quasi-Poisson manifold is a $G$-manifold $M$ equipped with a $G$-invariant bivector $\pi$ which satisfies
	\begin{equation} \label{eq:quasi-Poisson-Jacobi} [\pi, \pi]/2 =  \phi_M  \in \Gamma (\wedge^3 TM), \end{equation}
	where the right hand side is the action of the Cartan element $\phi\in \wedge^3 \g$ on the $G$-manifold $M$.
	\smallskip
	
	If $(M, \pi)$ is a $G\times G \times H$-quasi-Poisson manifold, then the \emph{fusion} \cite[]{AKSM} of $M$ is the $G\times H$-quasi-Poisson manifold $(M, \pi_\text{f})$ where $G$ acts diagonally, and the bivector $\pi_\text{f}$ is equal to
	\begin{equation}\label{eq:fusion-of-bivectors} \pi_\text{f} = \pi - \tfrac 12 t^{ij} (e_i)_1 \wedge (e_j)_2. \end{equation}
	Here $e_i$ is a basis of $\g$, $t = t^{ij} e_i \otimes e_j$ and the subscripts $1, 2$ denote the two $G$ actions on $M$.
\end{definition}

In terms of functions, the defining equation \eqref{eq:quasi-Poisson-Jacobi} says that 
\begin{equation}\label{eq:quasi-Poisson-Jacobi-on-functions}
  \{ \{f, g\} , h \} + \text{ cycl. } = - \tfrac 14 f^{ijk} e_i(f) e_j (g) e_k(h),
\end{equation}
where $f^{ijk} = f^i_{ab} t^{aj} t^{bk}$ are the structure constants of $\g$ with indices lifted via $t$. The fusion term in \eqref{eq:fusion-of-bivectors} is equal to
\begin{equation} \label{eq:fusion-on-functions} -\tfrac 12 t^{ij} \left[(e_i)_1(f) (e_j)_2 (g) - (e_i)_1(g) (e_j)_2(f)\right]. \end{equation}

\begin{remark}[Terminology]
	The unframed moduli space is often called the \emph{character variety} (since functions on it are conjugation invariant), while the framed moduli space is called the \emph{representation variety} (see e.g. \cite{BZBJ:integrating}). In terms of the category of quasicoherent sheaves on the character \emph{stack}, the functions on the character variety are given by the endomorphism algebra of the structure sheaf, while the functions on the representation variety are given by the \emph{internal} endomorphism algebra of the structure sheaf \cite[Def.~3.9]{BZBJ:integrating}. These two algebras (and their quantizations) live in different categories, which accommodates a relaxation of the Jacobi condition \eqref{eq:quasi-Poisson-Jacobi}.
\end{remark}

The group $G^V \cong G^{\times |V|}$ acts on the moduli space $M_{\Sigma, V}(G)$ by shifting the framing. In terms of holonomies, this is \[
((g_v)_{v\in V} \triangleright \rho ) (\gamma) = g_{t{(\gamma)}}\rho(\gamma) g_{s{(\gamma)}}^{-1}, \quad \text{ i.e.\ } \quad  \begin{overpic}{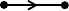} \put(-10, 25){$R_{g^{-1}}$} \put(86, 25){$L_{g}$}\end{overpic}, \]
where $\rho\colon \Pi_1(\Sigma, V) \to G$ is the holonomy of a flat connection and  $\gamma$ is a homotopy class of paths from $s(\gamma)$ to $t(\gamma)$. 
\begin{theorem}[{\cite{FockRosly, AMM, AKSM, LBS:poisson}}]
	There exists a unique family of $G^V$-quasi-Poisson bivectors on $M_{\Sigma, V}(G)$, for each $(\Sigma, V)$, satisfying:
	\begin{itemize}
		\item $\pi=0$ on the disk with two marked points on the boundary,
		\item if we locally modify the surface by joining vertices $v, v' \in V$ as follows
		\[  \begin{overpic}{pic/fusion-of-surfaces} \put(47, 15){$\mapsto$} 
		\put(10, 15){$v$}
		\put(18, 15){$v'$}
		\put(81.3, 26){$v''$}
		\end{overpic}  \]
		then the bivector field changes by the fusion term \eqref{eq:fusion-of-bivectors}, with the two $G$ actions coming at the action at $v$ and $v'$. The position of the new point $v''$ is determined by the orientation of the surface, which is taken to be counterclockwise on the above picture.		
	\end{itemize}
	Moreover, the induced \emph{Poisson} bracket on $G^V$-invariant functions on $M_{\Sigma, V}(G)$ agrees with the Atiyah--Bott--Goldman Poisson bracket.
\end{theorem}
To get a formula for $\pi$, we have to represent the surface as a gluing of disks. This theorem implies that the bivector $\pi$ depends only on $\Sigma$ and $V$ and not on this combinatorial data.

\subsection{Skeletons and spines}
If we connect the two marked points on the boundary of a disk by a segment, each gluing of the surface from such disks gives us a collection of segments such that $\Sigma$ retracts to the union of these segments. 
\begin{definition}[{see e.g. \cite{LBS:poisson}}]
Let $(\Sigma, V)$ be a marked surface. A skeleton is an isotopy class of embedded oriented graphs $\Gamma_\text{sk}  \hookrightarrow \Sigma$ such that vertices of the graph are mapped bijectively to $V$, and such that $\Sigma$ retracts to the embedded graph. No compatibility of the orientation of $\Gamma_\text{sk}$ and $\Sigma$ is required.
\end{definition}
The following are all skeletons, of five disks, an annulus and a punctured torus, respectively
\[\includegraphics{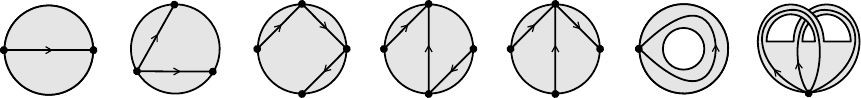}. \]

Each skeleton gives an identification between the moduli space and $G^{\text{edges}(\Gamma)}$. To incorporate information about the order of gluing and allow for taking $G$-invariant function on moduli spaces (the so-called reduction), we will work with a more general parametrization of moduli spaces induced by \emph{spines}.

In order to extend our results to character varieties of punctured surfaces, we allow spines of surfaces without marked points; we still require the boundary of each component of $\Sigma$ to be non-empty.
\begin{definition}
	A spine of a marked surface $(\Sigma, V)$, or of a surface $\Sigma$ with non-empty boundary of each component, is an isotopy class of embedded oriented graphs $\Gamma \hookrightarrow \Sigma$ such that 
	\begin{itemize}
	\item the vertices of $\Gamma$ are split into two disjoint subsets; the internal and external vertices,
	\item the internal vertices are mapped into the interior of $\Sigma$,
	\item the external vertices are mapped bijectively to $V$,
	\item the surface retracts to the embedded graph.
	\end{itemize}
	A uni-trivalent spine is a spine where the external vertices are univalent and the internal vertices are trivalent.
\end{definition}
Here are some examples of uni-trivalent spines.
\[ \includegraphics{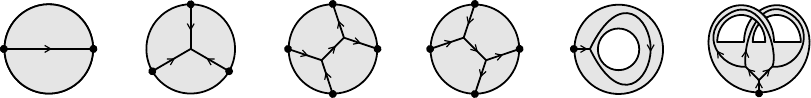} \] 

Each spine $\Gamma$ identifies the moduli space with the set $G^{\text{edges}(\Gamma)}/G^{\text{internal vertices}(\Gamma)}$. If each component of $\Sigma$ contains a marked point, as we assume, this quotient is a Cartesian power of $G$, i.e.\ a smooth manifold. Changing the orientation of an edge takes the corresponding holonomy to its inverse, and exchanges the Lie group actions $R_{g^{-1}}$ and $L_g$; hence each half-edge specifies a $G$-action even without a choice of orientation.\footnote{We could alternatively employ the trick from \cite{FockRosly} and use unoriented graphs; a holonomy along an unoriented edge is a map from the set of two orientations of the edge $\{\leftarrow, \rightarrow \} \to G$ such that the two values are related by taking the inverse.} Because of this, we will often drop the orientations of the edges of $\Gamma$.
\medskip

We will need the following result, due to Penner, which relates uni-trivalent spines via flips. This result also holds for spines without any external vertices.
\begin{theorem}\label{thm:flips}
A flippable edge $e$ is an edge of a uni-trivalent spine which connects two distinct trivalent points. The flip at $e$, denoted $\text{\Flip}(e)$ is the local move on the spine $\Gamma$
\begin{equation}\label{pic:flip}
	\raisebox{-0.4\height}{\begin{overpic}{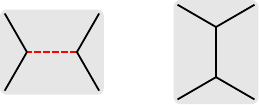}
		\put(50, 18){$\mapsto$}
		\put(19, 22){$e$}
			\put(76, 18){$e'$}	
	\end{overpic}}\;.
\end{equation}

All uni-trivalent spines are connected by a finite sequence of flips. Any two such sequences of flips are related by compositions of relations of the following form.
\smallskip

\textbf{involutivity}: $\Flip(e')\circ \Flip(e) = \textnormal{id}$ where $e'$ is the new edge created by $\Flip(e)$.
\smallskip

\textbf{naturality:} For two flippable edges $e_1, e_2$ which don't share a vertex, \[\Flip(e_1)\circ \Flip(e_2) = \Flip(e_2)\circ \Flip(e_1).\]

\textbf{pentagon:} For two flippable edges that share one common vertex, the following loop of flips is trivial. The red edge specifies the edge along which we flip when reading the circle clockwise.
\begin{equation}
\raisebox{-0.4\height}{\includegraphics{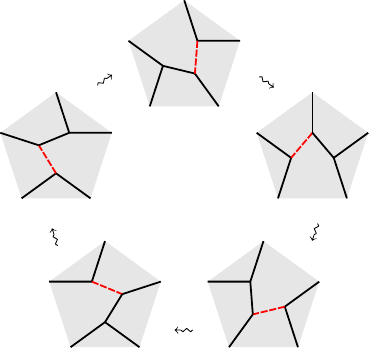}}
\end{equation}

\end{theorem}
\begin{proof}
Let us consider the following cell complex, the so-called \emph{fatgraph complex}. Its zero-cells are uni-trivalent spines. Its 1-cells are uni-trivalent spines  with one four-valent vertex, connecting two spines related by a flip. The 2-cells are spines with one five-valent vertex of two four-valent vertices, etc.. Penner proves in \cite[Chapter~4, Theorem~5.21]{Penner:DTT} that this cell complex comes from a cell decomposition of the Teichm\"uller space $\tilde{T}_*(\Sigma)$. In more detail and using notation of \cite[Chapter~4, Section~5.3]{Penner:DTT}, we take $P$ to contain all of the punctures (boundary components without marked points) of $\Sigma$ and take $Q$ to be equal to the set of midpoints of the boundary segments between vertices of $V$. Because there are no unmarked punctures, we don't have to include Penner's quasiflips. The space $\tilde{T}_*(\Sigma)$ is contractible by \cite[Chapter~2, Theorem~2.27]{Penner:DTT}.

Thus, any uni-trivalent spines can be connected via a path of flips, and any two such paths are connected by the relations as above. The involutivity relation comes from considering oriented paths in this complex.
\end{proof}
The fundamental groupoid of the above cell complex is the so-called Ptolemy groupoid, i.e.\ it has morphisms generated by flips and relations as above. The above theorem tells us that there is a unique morphism connecting any two objects, which are given by uni-trivalent spines. 
\subsection{Kontsevich integral}
The Kontsevich integral is an invariant of tangles with values in the vector space of chord diagrams. These chord diagrams act naturally on representations of Casimir Lie algebras, which is our motivation for their study: the star products we construct will be given by actions of chord diagrams coming from tangles.

\subsubsection{Chord diagrams}
Let us recall the definition of the space of chord diagrams, following \cite{LeMurakami:universal, LeMurakami:parallel}. See also the paper by Bar-Natan \cite{BarNatan:vassiliev} and the textbook of Chmutov, Duzhin and Mostovoy \cite{CDM} for other expositions.
\begin{definition}
	Let $S$ be a disjoint union of oriented segments $I=[0,1]$ and circles $S^1$. A \emph{collection of $n$ chords} is a set of $n$ unordered pairs of points of $S$, all disjoint. Each pair represents a chord connecting these points. Consider the complex vector space of linear combinations of collections of chords on $S$, seen up to isotopy of the endpoints, and the four-term (4T) relation \cite[Fig.~1]{LeMurakami:universal}, \cite[Fig.~1]{BarNatan:vassiliev}
	\begin{equation}
	\label{eq:4T-relations}
	\raisebox{-0.4\height}{\begin{overpic}{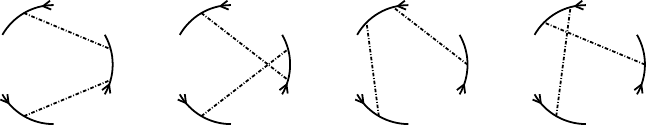}
	\put(21, 9.2){$-$}
	\put(48.6, 9.2){$=$}
	\put(76, 9.2){$-$}
	\end{overpic}}
	\end{equation}
	 This complex vector space is graded by the number of chords. Its completion with respect to this grading is the space of chord diagrams $\mathcal P(S)$.
	\end{definition}
We will always consider the connected components of $S$ to be linearly ordered. The case when $S$ consists of segments is the most important to us.
\begin{definition}
If $S$ is the disjoint union of $n$ intervals, we denote $\mathcal P_n = \mathcal P(I^{\sqcup n})$. We define an associative product on $\mathcal P_n$ given by $D\cdot D'$ by placing $D$ after $D'$ with respect to the orientation of the intervals. The unit of this algebra, given by the empty chord diagram, is denoted $1$. 	
\end{definition}
There is a considerable array of other operations on chord diagrams, of which we will need the following, see \cite{LeMurakami:parallel, LeMurakami:universal}
\begin{definition} \label{def:chord-diagrams-operations} Let $S$ and $S'$ be supporting 1-dimensional manifolds. We define the following linear maps. \smallskip

\textbf{disjoint union:} There is a map $ \mathcal P(S)\otimes \mathcal P(S') \to \mathcal P(S\sqcup S')$ given by taking the disjoint union of chords. It's denoted by the tensor product $\otimes$.
\smallskip

\textbf{minus chords:} There is a map $\mathcal P(S)\to \mathcal P(S)$ which multiplies each chord diagram by $-1$ if it contains an odd number of chords. We will denote it by $D \mapsto D_{-t}$.

\textbf{component reversal:} If $c\subset S$ is a connected component, then $S_c \colon \mathcal P(S) \to \mathcal P(S)$ is given by reversing the orientation of $c$ and multiplying the chord diagram by $-1$ if it has an odd number of chord endpoints on $c$.
\smallskip

\textbf{component removal:} If $c\subset S$ is a connected component, then $\varepsilon_c\colon \mathcal P(S) \to \mathcal P(S\setminus c)$ is given by removing $c$ and setting to zero any chord diagram which has a chord ending on $c$.
\smallskip

\textbf{component doubling or cabling:} If $c\subset S$ is a connected component, then $\Delta_c \colon \mathcal P(S) \to \mathcal P(S\sqcup c)$  is the map given by sending each chord diagram to the sum of all possible ways of distributing the chord endpoints on $c$ to the two copies of $c$ in $S\sqcup c$. 
\smallskip

\textbf{pushforward:} If $f\colon S \to S'$ is an orientation-preserving embedding, then $f_* \colon \mathcal P(S) \to \mathcal P(S')$ is given by pushing the diagram along $f$. If $f$ is a diffeomorphism $S\to S$, it's given by a permutation of the labels of $S$; in this case the map is denoted by $D \mapsto D^{\sigma(1) \dots \sigma(n)}$.

\end{definition}
All of these operations are compatible with the 4T relations.  The map $\Delta_c$ is coassociative in the obvious sense; we denote its iterated applications, which replace $c$ with $k+1$ copies of $c$, by $\Delta_c^k$. 

On $\mathcal P_n$, we denote by $\Delta_i$ the map which doubles the $i$th strand. For $\Delta_1 \colon \mathcal P_1 \to \mathcal P_2$, we use simply $\Delta$. On $\mathcal P_n$, the maps $\Delta_i$, $\varepsilon_i$ are algebra morphisms and the complete reversal $S_1 \dots S_n$ is a product-reversing involution.
\medskip

The following powerful operation on chord diagrams combines the doubling of strands and composition:
\begin{definition} Let $q\in c \subset S$ be a point on the component $c$ which is not a chord endpoint. For each number $r\in \mathbb N_{\ge 1}$ there is a map
\begin{align}\label{eq:insertion}  \text{ins}_c \colon \mathcal P_r\otimes \mathcal P(S) &\to \mathcal P (\Delta^{r-1}_c S)
\\
D_r \otimes D_S &\mapsto \text{ins}_c(D_r, D_S) \nonumber
\end{align}
given by taking $\Delta^{r-1}_c D_S$ and replacing the $r$ copies of a small interval around $q$ with $D_r \in \mathcal P_r$. 
\end{definition}\begin{proposition}[{\cite[Lemma~2.1]{LeMurakami:parallel}}]\label{prop:insert}
The insertion operation \eqref{eq:insertion} is independent of the choice of the point $q$.
\end{proposition}

We have the following information about the center of $\mathcal P_n$.
\begin{proposition}\label{prop:centerPn}
The image of $\Delta^{n-1}\colon \mathcal P_1 \to \mathcal P_n$ lies in the center of $\mathcal P_n$.
\smallskip

Let $t_{ij}\in \mathcal P_n$ be the chord connecting $i$-th and $j$-th strand.  
The element $\sum_{1\le i < j \le n} t_{ij} \in \mathcal P_n$ lies in the center of $\mathcal P_n$.
\end{proposition}
See \cite[Thm.~1]{LeMurakami:universal} for the first part, the $n=2$ case of the second part is used in \cite[Section~7]{LeMurakami:universal}. 
\begin{proof}
The first claim follows from the Proposition \ref{prop:insert}. The second claim is proven as follows. Apply $\Delta^{n-1}$ to $t_{11}\in \mathcal P_1$ to get
\[ \Delta^{n-1}(t_{11}) = \sum_i t_{ii} + 2 \sum_{i<j} t_{ij}.\]
 The LHS is in the center of $\mathcal P_n$, as is the first term of the RHS, by $r=1$ case of Proposition \ref{prop:insert}.
\end{proof}

\subsubsection{Combinatorial Kontsevich integral}
We will use the combinatorial Kontsevich integral defined on framed q-tangles, following definition of Le and Murakami from \cite{LeMurakami:parallel, LeMurakami:universal}. Let us briefly recall that a framed q-tangle is a framed  tangle whose endpoints are lying on two lines (bottom and top), together with a parenthesization of these two sets of endpoints. In pictures, we will encode this parenthesization by the spacing of the endpoints. The framing lies in the plane of the tangle, pointing to the right/left at positively/negatively oriented endpoints of the building blocks listed in Definition \ref{def:Kontsevich-integral} below. 

The Kontsevich integral depends on the choice of  a Drinfeld associator.
\begin{definition}
	Let $\Phi(X, Y)\in \mathbb C \langle \langle X, Y \rangle\rangle$ be formal non-commutative power series in two variables. Denote
	 \begin{equation} \label{eq:associator-t12-t23} \Phi  = \Phi(t_{12}, t_{23}) \in \mathcal P_3,\end{equation} then such $\Phi(X, Y)$ is a Drinfeld associator if
	\begin{align}
		\Phi(X\otimes 1 + 1 \otimes X, Y\otimes 1 + 1\otimes Y) &= \Phi(X, Y) \otimes \Phi(X, Y) \label{eq:assoc:grouplike} \\
		\Phi(Y, X) &= \Phi(X, Y)^{-1} \label{eq:assoc:inverse}\\
		\Phi(t_{23}, t_{34}) \cdot \Delta_2 \Phi \cdot \Phi(t_{12}, t_{23}) &= \Delta_3 \Phi \cdot \Delta_1 \Phi \quad \text{in $\mathcal P_4$}, \label{eq:assoc:pentagon} \\
		\Phi(t_{13}, t_{12}) \cdot e^{\tfrac 12 t_{13}} \cdot \Phi^{-1}(t_{13}, t_{23}) \cdot e^{\tfrac 12 t_{23}} \Phi &= e^{\tfrac 12 (t_{13}+ t_{23})}  \quad 
		\text{ in $\mathcal P_3$} 
	\end{align}
\end{definition}
	The first equation tells us that Drinfeld associator is grouplike in the universal enveloping algebra of the Drinfeld--Kohno Lie algebra $U(\mathfrak{dk}_3)$ \cite[Sec.~4]{Cartier}. Associators in $\mathcal P_3$ coming from non-commutative formal power series $\Phi(X,Y)$ are called horizontal, as they can be written using horizontal chords only (see e.g. \cite[Corr.~4.4]{BarNatan:braids}).
It is an important theorem of Drinfeld \cite{DrinfeldGal} that such associators exist, and can be chosen to have rational coefficients.
\begin{definition}\label{def:Kontsevich-integral}
The unnormalized Kontsevich integral $Z$ is defined by the following values on generators:
\[ Z \left(\raisebox{-0.3\height}{\,\includegraphics{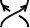}\,}\right) = e^{\tfrac 12 t_{12}}, \quad Z\left(\,\raisebox{-0.3\height}{\includegraphics{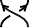}}\,\right) = e^{-\tfrac 12 t_{12}},\]
in the sense that we insert the chord diagram $e^{\pm \tfrac 12 t_{12}}$ below the crossing.

The cap and cup are sent to the chord diagram with no chords. For the associator, let $w_1$, $w_2$ and $w_3$ be three parenthesizations.
Then the straight tangle $(w_1 w_2)w_3$ to $w_1(w_2w_3)$ is sent by $Z$ to 
\[ Z\left(\;\raisebox{-0.3\height}{
\begin{overpic}{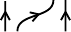}
	\put(-10, -15){\footnotesize{$w_1$}}
	\put(20, -15){\footnotesize{$w_2$}}
	\put(80, -15){\footnotesize{$w_3$}}
\end{overpic}}
\; \;\right) = \Delta^{|w_1|-1}_1 \Delta^{|w_2|-1}_2 \Delta^{|w_3|-1}_3\Phi \]
and 
\[ Z\left(\;\raisebox{-0.3\height}{
\begin{overpic}{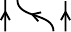}
	\put(-5, -15){\footnotesize{$w_1$}}
	\put(60, -15){\footnotesize{$w_2$}}
	\put(90, -15){\footnotesize{$w_3$}}
\end{overpic}}
\; \; \;\right) = \Delta^{|w_1|-1}_1 \Delta^{|w_2|-1}_2 \Delta^{|w_3|-1}_3\Phi^{-1},\]
where we use $\Phi\in \mathcal P_3$ from \eqref{eq:associator-t12-t23}. Here, each of the strands is, in general, composed of multiple parallel strands, according to $w_i$. 
\medskip

On more complicated tangles, $Z$ is defined by its compatibility with disjoint union and composition 
\[ Z(T_1 \sqcup T_2) = Z(T_1)\otimes Z(T_2) \quad \text{and} \quad Z(T \cdot T') = Z(T) \cdot Z(T').\]
Finally, if $c$ is a component of a tangle $T$ and $S_c T$ is the tangle $T$ with the orientation reversed, we declare
\begin{equation} \label{eq:Z-and-S} Z(S_c T) = S_cZ(T). \end{equation}
\end{definition}

This Kontsevich integral is well defined on tangle diagrams, but it's not invariant relative to the moves which change the number of critical points. The failure of being an isotopy invariant is measured by the following element.
\begin{definition}
The Kontsevich integral of the unknot is the chord diagram $\nu\in \mathcal P_1$ defined by
 \begin{equation} \label{eq:nu-definition}\nu^{-1}  = Z \left(  \raisebox{-0.4\height}{\includegraphics{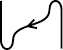}} \right) \in \mathcal P_1. \end{equation}
\end{definition}
The element $\nu$ does not depend on the associator $\Phi$, by an argument via twisting \cite[Remark~(1)]{LeMurakami:parallel}. The element $\nu$, and indeed any $Z(T)\in \mathcal P_n$ can be expanded as the empty chord diagram plus terms with at least one chord. On such elements $D = 1 + d$, we define real powers by the Taylor series
\begin{equation} \label{eq:power} D^s \coloneq 1 + s d + \tfrac 12 s(s-1) d^2 + \dots \end{equation}
which is well defined since only the first $k+1$ summands contribute terms with $k$ chords.

To get an isotopy invariant of framed q-tangles, we need to normalize $Z(T)$ as follows
\begin{definition}
The normalized Kontsevich integral is defined as the unnormalized one, with the exception that $\nu^{\tfrac 12}$ is assigned to the cap $\cap$ and the cup $\cup$.\footnote{In general, we could normalize with $\alpha, \beta\in \mathcal P_1$ such that $\alpha \beta = \nu$ and $S(\alpha) = \alpha$, $S(\beta) = \beta$.}
\end{definition}
It follows from \cite{LeMurakami:universal, LeMurakami:parallel} that both $Z$ and $\hat{Z}$ are well defined, isotopy invariants (for moves not changing the number of critical points for $Z$). Note that we use opposite multiplication for the algebras $\mathcal P_n$, and the strands go upwards for us by default -- at the end the maps defined here are the same as in \cite{LeMurakami:parallel}. We don't require the associator to be even. Our convention for the framing is slightly different from what is stated in \cite{LeMurakami:universal}, however we believe they use this convention to compute cabling in \cite{LeMurakami:parallel}. We also recall that $Z$ is compatible with strand reversal (by definition) and with strand removal (by a simple check on generators).

\subsection{Drinfeld categories}
A more categorical interpretation of the Drinfeld associator is the following construction of braided monoidal categories from infinitesimal deformations of symmetric monoidal categories, due to Cartier \cite{Cartier}. Let $k$ be a ring containing $\mathbb Q$ (or $\mathbb R$ or $\mathbb C$, depending on the coefficients of the Drinfeld associator we choose).
\begin{definition}
An infinitesimal braiding on a $k$-linear symmetric monoidal category $\Ca$ is given by a natural transformation $t_{X,Y}\colon X\otimes Y \to X\otimes Y$ such that $t_{Y, X} = \sigma_{X, Y} \circ t_{X, Y} \circ \sigma_{Y, X}$, where $\sigma$ is the symmetry of $\Ca$, and such that 
\[ \sigma_{X, Y} \circ \left(  1_{X\otimes Y} + \tfrac{\varepsilon}{2} t_{X, Y} \right) \colon X \otimes Y \to Y \otimes X \]
defines a braiding on\footnote{The notation of tensoring a category with a ring is a shorthand for tensoring the Hom-spaces of $\Ca$ with this ring.} $\Ca\otimes_k k[\varepsilon]/(\varepsilon^2)$ (i.e.\ the hexagon equation is taken only to the first order in $\varepsilon$). 
\end{definition}
For us, the most important example of such $\Ca$ is the $\mathbb{C}$-linear category $U\g\text{-mod}$ of complex representations of $\g$, with $t$ given by the action of a Casimir element. A Drinfeld associator can be used to quantize an infinitesimal braiding into a braiding as follows.
\begin{definition}
Let $\Ca$ be a category with an infinitesimal braiding and $\Phi$ a Drinfeld associator. Then we denote $\Ca_\hbar^\Phi$ the category $\Ca\otimes_k k[[\hbar]]$ with braiding given by $\beta_{X, Y} = \sigma_{X, Y} \circ \exp(\hbar t_{X, Y}/2)$ and the associator given by precomposing the associator of $\Ca$ with $\Phi(\hbar t_{X, Y}, \hbar t_{Y, Z})$.
\end{definition} 

We can consider associative algebras internal to the monoidal category $\Umodnp{\g}$. Since the associativity constraint is deformed using the associator, such algebras are a natural candidate for deformation quantization of quasi-Poisson manifolds.
\begin{definition}[{\cite[Def.~4.4]{EnriquezEtingof}}]
	A deformation quantization of a $G$-quasi-Poisson manifold $M$ is an associative algebra $A, m_0 + \hbar m_1 + \dots$ internal to $\Umodnp{\g}$ such that $(A, m_0)$ is the commutative algebra $C^\infty(M, \mathbb C)$ of complex-valued smooth\footnote{If we choose a real associator, we can restrict to real functions. Since the \cite{LBS:DQ} star products on moduli spaces are given by chord diagrams, we could also restrict to algebraic functions if $G$ is algebraic. For the same reason, the components $m_i$ are bidifferential operators.} functions on $M$ and 
	\[ m_1(f, g) - m_1(g, f) = \{f, g\}.  \]
\end{definition}

Finally, we can connect chord diagrams to Lie theory for Casimir Lie algebras as follows.
\begin{definition}
	Let $\g$ be a Casimir Lie algebra. The weight of a chord diagram $D \in \mathcal P_n$ with $k$ chords is the element $w_\g(D)\in (U\g)^{\otimes n}[[\hbar]]$ given by 
	taking $(\hbar t)^{\otimes k}\in \hbar^{k}\g^{\otimes 2k}$ and multiplying the $2k$ $\g$-factors in the order specified by the order of chord endpoints on the $n$ strands (see e.g. \cite[Sec~7]{LeMurakami:parallel}). 
\end{definition}
The weight $w_\g(D)$ is invariant with respect to the diagonal $\g$-action and gives a Hopf-algebraic interpretation to many of the operations from Definition \ref{def:chord-diagrams-operations}.
\medskip

A q-tangle $T$ without critical points can be decomposed into braidings and associators. It thus gives a natural transformation between repeated tensor product functors $(\Umodnp{\g})^{\otimes n} \to \Umodnp{\g}$. By construction of $Z$, the component of this natural transformation on a $\g$-representation $X_1 \otimes \dots \otimes X_n$ is given by the action $w_\g Z(T) \in (U\g)^{\otimes n}[[\hbar]]^\g$. We will often say that a tangle, or a chord diagram, act on a tensor product of $\g$-representations, by which we mean the action of its weight. \emph{Translating between action of $(U\g)^{\otimes n}$ and tangles will be the basis of most of our arguments about star products.}
\medskip

An example of such action is the monoidal structure on the \emph{fusion} functor, which quantizes the quasi-Poisson fusion \cite{LBS:DQ} (see Theorem \ref{thm:LBS-quantization} below).
\begin{definition}[{\cite[Prop.~2, Thm.~2]{LBS:DQ} }] \label{def:fusionreduction}
Let $\g$ and $\h$ be Casimir Lie algebras. Then \textit{fusion} is the (strong) monoidal functor $T\colon \Umod{\g\oplus\g\oplus\h} \to \Umod{\g\oplus\h}$ given by the pullback along the diagonal map $\g\to \g \oplus \g$. Its monoidal structure is the map 
$J_{X, Y} \colon X\otimes Y \to X\otimes Y$ given by the action of (the weight of) the following Kontsevich integral
\begin{equation}\label{eq:J}  J  \coloneq Z \left( \raisebox{-0.4\height}{\includegraphics{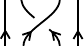}} \right) \in \mathcal P_4.\end{equation}
The strands are enumerated according to their order on the bottom line. 

\textit{Reduction} is the lax monoidal functor $\g\text{-inv} \colon \Umod{\g\oplus\h} \to \Umodnp{\h}$ given by taking invariants, with the inclusion $X^\g \otimes Y^\g \hookrightarrow (X\otimes Y)^{\g}$ as the monoidal structure.
\end{definition}
Since monoidal functors send associative algebras to associative algebras, the functors of fusion and reduction will be used to construct non-trivial algebras in $\Umodnp{\g}$.

\section{New results on Kontsevich integral}
Before discussing quantizations, we introduce a couple of new techniques related to the Kontsevich integral. These will be used in the next section; the reader interested in results is invited to skip to the next section and come back as necessary.

We start with a simple argument, which is a chord-diagrammatic version of the statement that for $e \in \g$, $e\otimes 1 + 1 \otimes e$ acts by zero on $(X\otimes Y)^\g$.
\begin{proposition} \label{prop:action-on-invariants}
Let $c$ be a component of $S$ diffeomorphic to the interval. Define a map $\mathcal P(S) \to \mathcal P(S)$ by first applying $\Delta_c$ (creating a new support interval $c'$), using the reversal operation $S_{c'}$ on the newly added component, and connecting the end of $c'$ to the beginning of $c$ using a cap, i.e.\
\[  \raisebox{-0.4\height}{\begin{overpic}{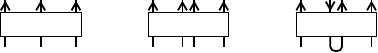}
	\put(26, 4){$\mapsto$}
	\put(62, 4){$\mapsto$}
	\put(9.5, -3){$c$}
	\put(46, -3){$c'$}
	\put(51, -3){$c$}
\end{overpic}} .\]
Then the result is equal to applying $\varepsilon_c$ followed by adding a component $c$ without any chord endpoints.
\end{proposition}
\begin{proof}
If there are no chord endpoints on $c$, the result is clear. If there are some chord endpoints, let us look at two terms of $\Delta_c$ which differ only by the position of the chord endpoint which was the closest to the beginning of $c$. After applying $S_{c'}$ and connecting, these two terms differ by a sign, and thus cancel.
\end{proof}
\medskip

If we consider the action of $D\in \mathcal P_n$ on diagonal invariants $(X_1 \otimes \dots \otimes X_n)^\g$, where each $X_i\in U\g\text{-mod}$, we can use the invariance to rewrite the action in terms of $n-1$ modules only. The chord--diagrammatic version is the following.

\begin{proposition}\label{prop:operation-r-on-chord-diagrams}
Let $r_i\colon \mathcal P_n \to \mathcal P_{n-1}$ be the map given by reversing the $i$-th segment, applying $\Delta^{n-2}_i$ to it and connecting thus created $n-1$ segments to the original $n-1$ segments. Then $r_i$ is an algebra morphism.
\end{proposition}
For example, $r_2 \colon \mathcal P_3 \to \mathcal P_2$ is equal to
\[  \raisebox{-0.4\height}{\begin{overpic}{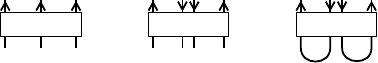}
	\put(28, 8){$\mapsto$}
	\put(67, 8){$\mapsto$}
\end{overpic}} .\]
\begin{proof}
In $r_i(D \cdot D')$, we can move all the chord endpoints of $D'$ to the cabling of the $i$-th strand. Then, by Proposition \ref{prop:insert}, this slides through $\Delta_i^{n-2}(D)$. In pictures,
\[ \raisebox{-0.4\height}{\begin{overpic}{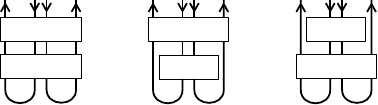}
	\put(28, 12){$=$}
	\put(68, 12){$=$}
	\put(8, 8){\footnotesize$D'$}
	\put(8, 17.7){\footnotesize$D$}
	\put(43, 8.5){\footnotesize{$r_2(D')$}}
	\put(47.5, 17.7){\footnotesize$D$}
	
	\put(83, 18){\footnotesize$r_2(D')$}
	\put(87, 8.5){\footnotesize$D$}
\end{overpic}} .\]
\end{proof}

\subsection{Reflections}
We will now discuss the Kontsevich integrals of tangles related by reflections. Le and Murakami considered the compatibility of $Z$ with reflections for even associators, in \cite[Prop.~3.1]{LeMurakami:parallel}; see also \cite[Section~3.1]{CDM}. We extend their result to all associators.
\begin{proposition} \label{prop:reflections}
	Let $T$ be a tangle. 

	\begin{enumerate}
	\item Let $T^{\leftrightarrow}$ be the rotation of $T$ around a vertical axis. Let $f\colon T \to T^{\leftrightarrow}$ be the diffeomorphism between the underlying $1$-manifolds.
	Then 
	\begin{equation} \label{eq:rotation-around-vertical-axis} f_* Z(T) = Z(T^{\leftrightarrow}) \end{equation}
	 \item Let $T^{\updownarrow}$ be the reflection of $T$ with respect to a horizontal line, i.e.\ exchanging overcrossings and undercrossings. Let $g\colon T \to T^{\updownarrow}$ be the diffeomorphism of the underlying $1$-manifolds. Then 
	 \begin{equation} \label{eq:reflection-about-horizontal-axis} g_* Z(T) = Z(T^{\updownarrow})_{-t}  \end{equation}
	 where $D_{-t}$ is defined by multiplying the chord diagram $D$ by $-1$ if it contains an odd number of chords.
	\end{enumerate} 
	 \end{proposition}
	Note that the endomorphisms $f_*$, $g_*$ of $\mathcal P(S)$ permute the components of $S$ in general. For example, applying part 1 to the tangle $J$ from \eqref{eq:J} tells us that $J^{4321} = J$.
	\begin{proof}
		Part $1$ is proven in \cite[Prop.~3.1]{LeMurakami:parallel}. For part 2, we proceed by considering the theorem on generators of  tangles, as in \cite{LeMurakami:parallel}. The associator is the only non-trivial case. 
		The map $g_*$ on the left hand side of \eqref{eq:reflection-about-horizontal-axis} is the identity, and the left hand side is therefore given by taking the reversed order of the multiplication $\Phi^\text{rev}(X, Y)$. 
		Since the associator is grouplike \eqref{eq:assoc:grouplike} in the universal enveloping algebra $U\mathfrak{dk_3}$, the antipode sends $\Phi(X,Y)$ to its inverse, i.e.\
\begin{equation} \label{eq:assoc:reverse} \Phi^{\text{rev}}(-X, -Y) = \Phi^{-1}(X, Y),\end{equation}
i.e.\ the left hand side of \eqref{eq:reflection-about-horizontal-axis} is 
\[  \Phi^\text{rev}(X, Y) = \Phi^{-1}(-X, -Y)\]
This is exactly the right hand side of \eqref{eq:reflection-about-horizontal-axis}.
	\end{proof}
	
	A simple consequence is the following, which was implicitly used by \cite{LeMurakami:parallel}.
	\begin{proposition}\label{prop:S-nu}
	Recall that $\nu\in \mathcal P_1$ is the so-called Kontsevich integral of the unknot. Then
	\[ S(\nu) = \nu\,. \]
	\end{proposition} 
	\begin{proof}
	If we rotate the tangle defining $\nu$ in \eqref{eq:nu-definition} around a vertical axis, the map $f_*$ is the identity on $\mathcal P_1$. Applying Equation \eqref{eq:rotation-around-vertical-axis}, we get that 
	\[ \nu^{-1} = Z \left(  \raisebox{-0.4\height}{\includegraphics{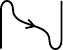}}\right). \]
	A further reflection around a horizontal line again gives $g_*$ equal to the identity on $\mathcal P_1$ and thus we have
	\[ \nu^{-1} = Z \left(  \raisebox{-0.4\height}{\includegraphics{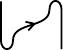}}\right)_{-t}. \]
	Since $\nu$ is independent of the associator, we can compute it using an even associator, for which each term has an even number of chords. Thus, multiplying each chord by $-1$ has no effect. 
	By \eqref{eq:Z-and-S}, the RHS is equal to $S(\nu^{-1})$. What remains to be shown is that we can invert the equality $\nu^{-1} = S(\nu^{-1})$. A priori, $S\colon \mathcal P_1 \to \mathcal P_1$ is an algebra anti-automorphism. However, $\mathcal P_1$ is commutative (\cite[Theorem 7]{BarNatan:vassiliev} or Proposition \ref{prop:insert}), and thus $S$ is an algebra automorphism. Using \eqref{eq:power}, we get $S(\nu^{-1}) = S(\nu)^{-1}$.
	\end{proof}
	
\subsection{Caps, cups and cabling}
The following elements $a, b\in \mathcal P_2$ were introduced by \cite[Section~4]{LeMurakami:parallel}, and will play a key role in our description of the deformation quantization. These elements are related to the isomorphism $(V\otimes W)^* \cong W^* \otimes V^*$ of finite-dimensional $G$-representations \cite[Proposition~1.2]{Drinfeld:quasi-hopf-algebras}. They were also used to relate horizontal reflections of tangles in \cite[Lemma~8.5]{HabiroMassuyeau:handlebodies}, under the name of cabling anomaly; and in \cite{Dancso:trivalent} to define a Kontsevich integral of trivalent graphs.
\begin{definition}
	Define $a, b \in \mathcal P_2$ by\footnote{A useful mnemonic is that the word cap contains the letter ``a''.}
	\begin{equation}\label{eq:cap-and-cup-definition}
		a \coloneq Z\left( \raisebox{-0.30\height}{\includegraphics{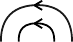}} \right) \quad \text{ and } \quad b \coloneq Z\left( \raisebox{-0.45\height}{\includegraphics{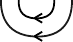}} \right).
	\end{equation}
\end{definition}
For concreteness, the first strand of $a,b \in \mathcal P_2$ is the outer strand of the cap and cup; however $a, b$ don't depend on this choice. We list this and other useful properties of $a$ and $b$; extending \cite{LeMurakami:parallel, Dancso:trivalent}.
\begin{proposition}\label{prop:properties-of-a-and-b}
	The elements $a$ and $b$ satisfy these properties.
	\begin{itemize}
	\item \textbf{strand removal:}
	\[ \varepsilon_1(a) = \varepsilon_2(a) = \varepsilon_1(b) = \varepsilon_2(b)  = 1 \in \mathcal P_1 \]
	\item \textbf{strand switching:}
	\[ a^{21} = a, \quad b^{21} = b, \]
	
	\item \textbf{orientation reversal:}
	\begin{equation}\label{eq:Sab}  a = S_1 S_2 (a), \quad b = S_1 S_2 (b), \end{equation}

	\item \textbf{reflection:}
	\begin{equation} \label{eq:ab-minus-t} a_{-t} = b ,\end{equation} 
	
	\item \textbf{action on invariants}
	\begin{equation}\label{eq:ri-of-ab} r_2(a) = r_2(b) = r_1(a) = r_1(b) = \nu^{-1},\end{equation}
	
	\item \textbf{product:}
	\begin{equation}\label{eq:product-of-ab} ab = ba = \Delta \nu \;\cdot\; \nu^{-1} \otimes \nu^{-1}, \end{equation}
	
	\item \textbf{expansion:} $a$ and $b$ are of the form 
	\[ 1 + \text{ terms in chord degree $2$ and more}, \]
	
	\item\textbf{even associators:} If the associator is even ($\Phi(-X, -Y) = \Phi(X, Y)$), then
	\begin{equation} \label{eq:ab-even-associator} a = b = \Delta \nu^{\tfrac 12} \;\cdot  \; \nu^{-\tfrac 12} \otimes \nu^{-\tfrac 12}. \end{equation}
	\end{itemize}
\end{proposition}
\begin{proof}{}
\begin{itemize}
	\item \textbf{strand removal:}
	Follows from the compatibility of the strand removal with the Kontsevich integral.
	\item \textbf{strand switching:}
	This follows from the isotopy of tangles
	\[ \raisebox{-0.4\height}{\includegraphics{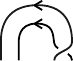}} \quad \sim \quad \raisebox{-0.4\height}{\includegraphics{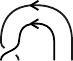}} \]
	on which the Kontsevich integral gives
	\[ R \cdot a = a^{21} \cdot R\, \]
	but $R = e^{t_{12}/2}$ is central by Proposition \ref{prop:centerPn} 
	
	\item \textbf{orientation reversal:}
	Follows from applying the first part of Proposition \ref{prop:reflections} to the doubled cap and cup. Appeared in \cite[below~Lemma~6.2]{Dancso:trivalent}.

	\item \textbf{reflection:}
	Follows from applying the second part of Proposition \ref{prop:reflections} to the doubled cup.

	\item \textbf{action on invariants} We evaluate the equality
	\[  \hat{Z} \left(\raisebox{-0.4\height}{\includegraphics{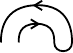}} \right) = \hat{Z}\left(\raisebox{-0.3\height}{\includegraphics{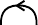}} \right) \]
	The left hand side is $\nu^{\tfrac 32}\cdot r_1(a)$, while the right hand is simply given by $\nu^{\tfrac 12}$. The other equations are proven similarly, or using the previous two points. Appeared in \cite[proof of Prop.~6.5]{Dancso:trivalent}
		
	\item \textbf{product:}
	Follows from computing the Kontsevich integral of cabling of the ``snake'', see \cite[{\textit{Important remarks}, Sec.~4}]{LeMurakami:parallel}

	\item \textbf{expansion:} The associators, of which $a$ and $b$ are built of, are both of the form $1+ \text{ terms with $2$ or more chords}$.
	
	\item \textbf{even associators:} See \cite[Lemma~4.2]{LeMurakami:parallel}.
	\end{itemize}
\end{proof}

\subsection{Cabling}
Finally, let us explain how the Kontsevich integral commutes with the operation of cabling, i.e.\ doubling a strand. This is a mild generalization of the result of Le and Murakami \cite[Thm.~4.2]{LeMurakami:parallel}.
\begin{proposition}
	Let $T$ be a tangle, $c\subset T$ its connected component and $\Delta_c T$ obtained by doubling the component $c$ and modifying the parenthesization such that the two new endpoints are parenthesized together. Then, using the insertion operation from Definition~\ref{def:chord-diagrams-operations}
	\begin{equation}\label{eq:unnormalized-cabling} Z(\Delta_c T) = \operatorname{ins}_c ( a^{m_+} \cdot b^{m_-}, Z(T) )\,.\end{equation}
	Here, $m_+$ and $m_-$ are the numbers of maxima and minima on the component $c$ the tangle.
	If the associator is even, or $m_+ = m_-$, then for the normalized Kontsevich integral we have
	\[ \hat{Z}(\Delta_c T) = \Delta_c \hat{Z}(T)  \]
\end{proposition}
\begin{proof}
The proof is completely analogous to \cite[Thm.~4.2]{LeMurakami:parallel}, i.e.\ it is proven by considering the equality on generators. The cap and the cup contribute the terms $a$ and $b$ to the left hand side.
\end{proof}
Let us recall the notion of twisting of associators (e.g. \cite{LeMurakami:universal}). Let $F\in \mathcal P_2$ be such that $F^{21} = F$ and $\varepsilon_1(F) = \varepsilon_2(F) = 1 \in \mathcal P_1$. Then the  associator $\Phi$ twisted by $F$ is defined as
\[  \Phi^F \coloneq \Delta_2 F \;\cdot \; 1 \otimes F \;\cdot \; \Phi \;\cdot \; F^{-1} \otimes 1 \;\cdot \; \Delta_1 F^{-1}. \]
A twisted associator can again be used to define a Kontsevich integral. 
\begin{proposition} \label{prop:twist-by-phi}
	The twist of $\Phi\in\mathcal P_3$ by $a\in \mathcal P_2$ satisfies
	\begin{equation}
	\Phi^a = \Phi_{-t}.
	\end{equation} 
	It follows that there exists $\mu \in \mathcal P_1$ such that $a = \Delta \mu \cdot \mu^{-1} \otimes \mu^{-1}$ if and only if the associator is even.
\end{proposition}
\begin{proof}
Let us compute the normalized Kontsevich integral of the following isotopy
\[\begin{overpic}{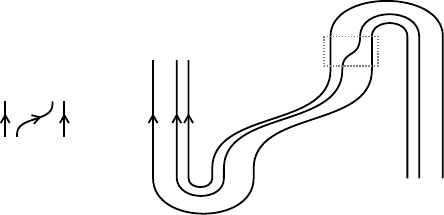}
\put(22, 20){$\sim$}
\end{overpic}\]
The left hand side is the usual associator. The right hand side has terms
\begin{align*}
 \Delta_1 a  \cdot a \otimes 1 & \quad \text{from the cap with one doubled strand, using \eqref{eq:unnormalized-cabling}}\\
\Phi & \quad\text{from the associator in the dotted rectangle}\\
 \Delta_2 b  \cdot 1 \otimes b & \quad\text{from the cup with one doubled strand, using \eqref{eq:unnormalized-cabling}} \\
  \nu \otimes \nu \otimes \nu & \quad\text{from the normalization of caps and cups,}\\
\Delta^2(\nu^{-1}) &\quad\text{from the cabling of the repeated associator.}
\end{align*}
The last two terms commute with everything, and the first three terms can be separated from them by Proposition \ref{prop:insert}. Let us move everything to the bottom right corner of the RHS.
Transferring the chords there has a non-trivial effect on the associator $\Phi(X=t_{12},Y=t_{23})$ in the dotted rectangle: it switches the role of $X$ and $Y$, and completely reverses the order of multiplication. By \eqref{eq:assoc:reverse}, we have
\[ \Phi^{\text{rev}}(Y, X) = \Phi^{-1}(-Y, -X) \stackrel{\eqref{eq:assoc:inverse}}{=}  \Phi(-X, -Y).\]
The whole equality of tangles above gives
\[ \Phi = \Delta^2(\nu^{-1}) \; \cdot \; \nu \otimes \nu \otimes \nu \; \cdot \; \Delta_2 b  \; \cdot \; (1 \otimes b) \; \cdot \; \Phi_{-t} \;\cdot \; \Delta_1 a  \;\cdot \; (a \otimes 1) \]
 Moving the non-associator terms to the other side and expressing $b$ in terms of $a$ using \eqref{eq:product-of-ab}, we get
\[ \Delta_2 a \; \cdot \; (1 \otimes a) \;\cdot\; \Phi \;\cdot \; \Delta_1 a^{-1}  \;\cdot \; (a^{-1} \otimes 1) = \Phi_{-t}.  \]
which is what we wanted to prove.

The second claim follows from the fact that a twisting preserves an associator if the twist is of the form $\Delta \mu \cdot \mu^{-1} \otimes \mu^{-1}$ (Prop. \ref{prop:insert}).
\end{proof}

\section{Spines and quantizations}

We will now use Kontsevich integral to describe the star products constructed by Li-Bland and Ševera in \cite{LBS:DQ}. Their star product depends on decomposition of the surface into a gluing of disks, and the order of the gluings. We will do away with some of this dependence immediately. However, to proceed further, we will twist the fusion of \cite{LBS:DQ} which will reduce the necessary input data to a uni-trivalent spine. After that, we can go back to the original (untwisted) star-products of \cite{LBS:DQ} and characterize the additional decoration on spines needed to specify their star products. In both cases, we get explicit isomorphism between star products which are closed under composition.

\subsection{Fusion and reduction}
The star products in \cite{LBS:DQ} are constructed by gluing the marked surface $\Sigma$ from disks with two marked points on the boundary. If needed, reductions, which remove marked points on the boundary, are performed. On the level of star products, this corresponds to the application of the functors fusion and reduction from Definition \ref{def:fusionreduction}. \medskip

In the quantum case, there is an important choice of sign invisible on the semiclassical level, which we now describe.
Following \cite{LBS:DQ}, we fix a Casimir Lie algebra $(\g, t)$ and denote by $\bar{\g}$ the Casimir Lie algebra $(\g, -t)$. We will choose a marking of each $v\in V$ with a sign $\pm$, i.e.\ a decomposition of $V$ into $V_+ \sqcup V_-$. This marking determines the category in which the deformation quantization of  the moduli space of flat connections will live: the following theorem gives an iterative way to construct an algebra in 
\[ \Umod{\g^{V_+}\oplus \bar{\g}^{V_-}}. \]
\begin{theorem}[\cite{LBS:DQ}]\label{thm:LBS-quantization}
	Let $(\Sigma, V = V_+ \sqcup V_-)$ be a marked surface with marked points decorated by $\pm$. Let $M_{\Sigma, V}(G)$ be the corresponding moduli space of flat connections and $A\in \Umod{\g^{V_+}\oplus \bar{\g}^{V_-}}$ a deformation quantization of this space. 
	\begin{itemize}
		\item Let $(\Sigma_\text{f}, V_\text{f})$ be obtained by corner connected sum of two marked points in $V$ 		
		\[ \raisebox{-0.4\height}{\begin{overpic}{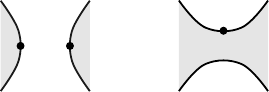}
				\put(9.5, 15){$+$}
				\put(17.5, 15){$+$}
				\put(80, 26){$+$}

		 \end{overpic}}\;, \quad \quad \quad \raisebox{-0.4\height}{\begin{overpic}{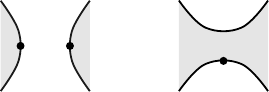}
						\put(9.5, 15){$-$}
				\put(17.5, 15){$-$}
				\put(80, 6.5){$-$}
		 \end{overpic}}\,.\]
		 Then $T(A)$, the \emph{fusion} of $A$, is a deformation quantization of $M_{\Sigma_\text{f}, V_\text{f}}$. Here, the two factors $\g$ on which fusion (and hence $J\in \mathcal{P}_4$) acts are specified by the ordered pair of the two marked points. There are two cases, depending on the sign of the marked points.

		\item If $(\Sigma, V\setminus \{v\})$ is obtained by removing a marked point $v\in V$ (of any sign),  then $A^\mathfrak{g}$, the \emph{reduction} of $A$ w.r.t the action at $v$, is a deformation quantization of $M_{\Sigma, V\setminus \{ v\}}$.
	\end{itemize}
\end{theorem}
\begin{proof}
Since our conventions differ from \cite{LBS:DQ}, let us prove the claim about fusion. To the first order, the product of two functions $f, g$ is modified by the action of $\tfrac 12 t_{f_2, g_1}$. After antisymmetrization, we obtain the fusion term of \eqref{eq:fusion-on-functions} with the minus sign there. The fusion at vertices decorated with minus acts via $-t$, and this additional sign can be absorbed by switching the order of the actions in the quasi-Poisson fusion, which corresponds to moving the point to the other boundary\footnote{This minus sign can also be implemented by using an element $J$ with an undercrossing instead. This gives isomorphic algebras, as the fusions are monoidally isomorphic via $e^{\tfrac 12 t_{12}}$.} 
\end{proof}

Finally, we need to specify a building block, which is given as follows.
\begin{proposition}[{\cite[Example~1]{LBS:DQ}}]\label{prop:building-block}
	The undeformed commutative product on $C^\infty(G, \mathbb{C}) \in \Umod{\g \oplus \bar{\g}}$ provides a deformation quantization of the (vanishing) $G^2$-quasi-Poisson structure on the moduli space of flat connections on the disk with two marked points on the boundary.
	\begin{equation}\label{eq:disk-pm}
	\raisebox{-0.4\height}{
	\begin{overpic}{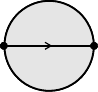}
		\put(-19, 41.6){$-$}
		\put(102, 41.6){$+$}
	\end{overpic}
	}
	\end{equation}
\end{proposition}

\paragraph{Quantization procedure of \cite{LBS:DQ}:} We start with $N$ building blocks \eqref{eq:disk-pm}, and then repeatedly fuse the marked points and possibly remove them by taking invariants. Applying the corresponding monoidal functors of fusion and reduction gives an element of $D\in \mathcal{P}_{4N}$ made of $J$'s, which acts on a pair of elements of $C^\infty(G^N, \mathbb{C})^{\g^\text{red}}\cong C^\infty(M_{\Sigma, V_+\sqcup V_-}(G), \mathbb{C})$. The star product of $f, g \in C^\infty(G^N,\mathbb{C})^{\g^\text{red}}$ is given by 
\[ f\star g = m_0 \circ w_{\g^{V_+}\oplus \bar{\g}^{V_-}}(D) (f\otimes g) \]
where we see $f\otimes g$ as an element in $C^\infty(G^N,\mathbb{C})\otimes C^\infty(G^N,\mathbb{C})$, with the $\g^{4N}$-action coming from the left and right-invariant vector fields on the $2N$ copies of $G$.  
\smallskip

\paragraph{Dependence on choices} \label{par:dependence-on-choices} This star product depends on how we decompose the surface $\Sigma$ into disks \eqref{eq:disk-pm}, as well as the order of the reductions and fusions. The decomposition of $\Sigma$ into disks can be equivalently encoded in a skeleton: each disk corresponds to one edge of the \emph{skeleton}. Since fusions and reductions at different vertices obviously commute, we only need to remember the order of operations at each vertex. This can be encoded by
\begin{itemize}
\item linearly ordering and parenthesizing the half-edges at each vertex; this specifies the order of the fusions. We will represent this data by resolving an $n$-valent vertex into a planar binary tree with $n$ leaves, as in \eqref{diag:two-ways-fourvertex} below;
\item additionally remembering if reduction (taking $G$-invariants) was performed at the vertex, which will be represented by placing marked point to the interior of the surface. In order to specify from which boundary component it came from, we connect it to this component by a dotted line, as in \eqref{diag:two-ways-twovertex}.\footnote{These dotted lines were called cilia by \cite{FockRosly}, but they used them for vertices where reduction was not performed.}
\end{itemize}
Our next goal is to show that one can forget some of the above information. First, we do away with the dotted line and the linear order at internal vertices. Then, we show that  one can move fusions between neighboring vertices, which will reduce the necessary information to a \emph{uni-trivalent spine}.

\subsection{The disk with two and three marked points: a case study}\label{sec:case-study}
As a warm-up for working with Kontsevich integral in this setting, we will describe deformation quantizations of disks with two and three marked points.

\paragraph{Disk with two marked points:}
If we want to construct a deformation quantization of a disk with the same sign on both marked points, we can proceed by gluing two disks \eqref{eq:disk-pm} together and reducing at the created vertex.
\begin{proposition} \label{prop:reduction-and-opposite-fusion} Let $T$ and $T^\text{op}$ be the functors $\Umod{\g^{\oplus 2}} \to \Umodnp{\g}$ given by the usual and reversed \footnote{This means that the role of the $\g$-actions is reversed, or in other words we act with $J^{2143}$.} fusion. Then the composed functors 
\begin{equation} \label{diag:T-Top-inv} \begin{tikzcd}
	{\Umod{\g^{\oplus2}}} & {\Umodnp{\g}} & {\mathrm{Vect}}
	\arrow["T", curve={height=-18pt}, from=1-1, to=1-2]
	\arrow["{T^\text{op}}"', curve={height=18pt}, from=1-1, to=1-2]
	\arrow["{\g\text{-inv}}", from=1-2, to=1-3]
\end{tikzcd}\end{equation}
agree as monoidal functors.
\end{proposition}
Applying the two functors to $C^{\infty}(G\times G,\mathbb{C}) \in \Umod{\g^{\oplus 2} \oplus \bar{\g}^{\oplus 2}}$, we get a (quantization of the moduli space of the) disk with two points on the boundary marked $-$ and with one internal bivalent vertex, coming from opposite sides of the disk.
\begin{equation}\label{diag:two-ways-twovertex}  \raisebox{-0.4\height}{\begin{overpic}{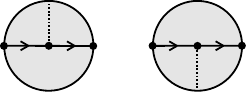}
	\put(-9, 16.7){$-$}
	\put(16.2, 10.5){$+$}
	\put(40, 16.7){$-$}
	\put(52, 16.7){$-$}
	\put(76.8, 22.5){$+$}
	\put(100.5, 16.7){$-$}
\end{overpic}}  \end{equation} 
Proposition \ref{prop:reduction-and-opposite-fusion} tells us that the two star products are equal.
\begin{proof}\!\!\footnote{See \cite[Proposition 4.1.1]{Jan:thesis} for an   alternative proof.}
	The top composite functor in \eqref{diag:T-Top-inv} has a monoidal structure given by the following map of $U\g$-modules\footnote{Note that the composition of monoidal constraints is ``in the opposite order'' relative to the composition of functors: for two monoidal functors $G\circ F$, the monoidal structure is $ G(\mu^F_{X, Y})\circ \mu^G_{FX, FY}\colon GFX \otimes GFY \to G(FX \otimes FY)\to GF(X\otimes Y) $.} 
	\[  X^\g \otimes Y^\g \hookrightarrow (X\otimes Y)^\g \xrightarrow{J_{X_1, X_2, Y_1, Y_2}} (X\otimes Y)^\g. \]
	Let us explain the notation, since we will be using in the following. Both $X$ and $Y$ are $U(\g^{\oplus 2})$-modules, and labels $X1, X2$ label the two $\g$ actions. The element $J\in \mathcal P_4$ determines a morphism $J_{X1, X2, Y1, Y2}\colon X\otimes Y \to X\otimes Y$ given by an action of $w_\g(J)$ via the four specified Lie algebra actions. 	

	The action of $J$ on $X^\g \otimes Y^\g$ can be expressed using only an action of $X_1$ and $Y_1$, by reversing the strands $X_2$, $Y_2$ and connecting them; in other words it's given by the action of the \emph{unnormalized} Kontsevich integral of the tangle on the right
	\vspace*{0.5em}
	\[ \raisebox{-0.5\height}{\begin{overpic}{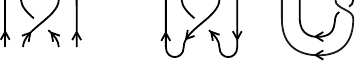}
	\put(1.5, -0.5){\makebox[0pt]{\scaleobj{0.75}{X_1}}}
	\put(7.5, -0.5){\makebox[0pt]{\scaleobj{0.75}{X_2}}}

	\put(63, 18){\makebox[0pt]{\scaleobj{0.75}{X_1}}}
	\put(69, 18){\makebox[0pt]{\scaleobj{0.75}{Y_1}}}
	
	\put(95, 18){\makebox[0pt]{\scaleobj{0.75}{X_1}}}
	\put(101, 18){\makebox[0pt]{\scaleobj{0.75}{Y_1}}}
	
	\put(32, 8){$\leadsto$}
	
	\put(71,8){$\sim$}
	
		\put(17, -0.5){\makebox[0pt]{\scaleobj{0.75}{Y_1}}}
	\put(23, -0.5){\makebox[0pt]{\scaleobj{0.75}{Y_2}}}
	\end{overpic}}. \]
	The isotopy above shows that the monoidal structure is given by $K_{X_1, Y_1}$, where
	\[ K \coloneq b \cdot e^{-t/2}, \]
	see \eqref{eq:cap-and-cup-definition}.
	The monoidal structure on the bottom functor in \eqref{diag:T-Top-inv} is given by the action
			\[  X^\g \otimes Y^\g \hookrightarrow (X\otimes Y)^\g \xrightarrow{J_{X_2, X_1, Y_2, Y_1}}(X\otimes Y)^\g. \]
	Expressing everything in terms of the actions $X_1, Y_1$, get that the monoidal structure is given by the tangle
	\[ \raisebox{-0.4\height}{\begin{overpic}{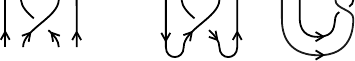}
	\put(1.5, -0.5){\makebox[0pt]{\scaleobj{0.75}{X_2}}}
	\put(7.5, -0.5){\makebox[0pt]{\scaleobj{0.75}{X_1}}}

	\put(63, 18){\makebox[0pt]{\scaleobj{0.75}{X_1}}}
	\put(69, 18){\makebox[0pt]{\scaleobj{0.75}{Y_1}}}
	
	\put(95, 18){\makebox[0pt]{\scaleobj{0.75}{X_1}}}
	\put(101, 18){\makebox[0pt]{\scaleobj{0.75}{Y_1}}}
	
	\put(32, 8){$\leadsto$}
	
	\put(71,8){$\sim$}
	
		\put(17, -0.5){\makebox[0pt]{\scaleobj{0.75}{Y_2}}}
	\put(23, -0.5){\makebox[0pt]{\scaleobj{0.75}{Y_1}}}
	\end{overpic}}\]
	In other words, the monoidal structure is given by
	\[ \tilde{K} \coloneq e^{-t/2}\cdot S_1 S_2(b) \]
	Since by Proposition \ref{prop:properties-of-a-and-b} the element $b$ is symmetric w.r.t the reversal of the both strands, and $t$ is in the center of $\mathcal P_2$, we get that $\tilde{K} = K$.
\end{proof}
Similarly, we can glue the points labeled with $-$ and reduce, giving a quantization of
\[  \raisebox{-0.4\height}{
\begin{overpic}{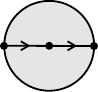}
	\put(-20, 41.5){$+$}
	\put(41, 51){$-$}
	\put(102,41.5){$+$}
\end{overpic}
}. \]  
Using \eqref{eq:ab-minus-t}, the quantization is given by the action of $K_{-t} = a \cdot e^{t/2}$. On a pair of functions $f, g \in C^\infty{(G, \mathbb{C})} \cong C^\infty (G\times G, \mathbb{C})^G$, the star product is given by
\[ m_0 \circ (a\cdot e^{t/2})^{L_f; L_g} (f\otimes g)  = m_0 \circ (a\cdot e^{t/2})^{-R_f; -R_g} (f\otimes g) ,  \]
where the weight $w_{\g}(a\cdot e^{t/2}) \in U\g\otimes U\g$ acts on the two functions by left or right invariants vector fields, as specified by the superscripts; these actions are equal since $w_{\g}(a\cdot e^{t/2})$ is $\Ad$-invariant.

\paragraph{Disk with three marked points}
We now consider fusing three disks into one, first without any reduction. First, we relate the two different parenthesizations of this double fusion, see also \cite[Thm~2.1]{LBS:coherent}.
\begin{proposition}\label{prop:two-fusions}
Let $T_{(3)}$ and $T'_{(3)}$ be the monoidal functors $\Umod{\g^{\oplus 3}} \to \Umodnp{\g}$ given by repeated fusion in left and right parenthesization. Then the natural isomorphism $T_{(3)} \Rightarrow T'_{(3)}$, with components given by the action of the associator $\Phi$, is monoidal.
\end{proposition}
Applying the functors $T_{(3)}$ and $T'_{(3)}$ to $C^\infty(G^3, \mathbb{C})\in \Umod{\g^{\oplus 3}\oplus \bar{\g}^{\oplus 3}}$, we get two possible deformation quantizations of the moduli space of flat connections on a disk with 4 marked points
\begin{equation} \label{diag:two-ways-fourvertex}  \raisebox{-0.4\height}{\begin{overpic}{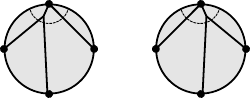}
\put(16.4,41){$+$}
\put(-7,18){$-$}
\put(40,18){$-$}
\put(16.4,-6){$-$}
\put(76.8,41){$+$}
\put(53.4,18){$-$}
\put(100.4,18){$-$}
\put(78.4,-6){$-$}
\end{overpic}}\quad . \end{equation}
The proposition above tells us that these quantizations are isomorphic, with the isomorphism given by the action of the Drinfeld associator.
\begin{proof}
	The monoidal structure $\mu_{X, Y}\colon T_{(3)}(X)\otimes T_{(3)}(Y) \to T_{(3)}(X\otimes Y)$ on functor $T_{(3)}$ can be written as	\[ J_{X_1, X2, Y1, Y2}\circ J_{X1X2, X3, Y1Y2, Y3}. \]
	Multiple actions in one index mean we act diagonally (or equivalently, we use the doubling operation $\Delta_i$ from Definition \eqref{def:chord-diagrams-operations} on the chord diagram $J$).

To check that the natural transformation with components  $\Phi_{X1, X2, X3}$ is monoidal, we will use that all of the morphisms come from chord diagrams which can be obtained via Kontsevich integral from parenthesized tangles. Indeed, the equality $ \Phi_{X1Y1, X2Y2, X3Y3} \circ \mu_{T_{(3)}(X), T_{(3)}(Y)}  = \mu'_{T_{(3)}'(X), T_{(3)}'(Y)} \circ \Phi_{X1, X2, X3} \otimes \Phi_{Y1, Y2, Y3}$ can be written as

\begin{equation} \label{eq:JJ1fusion}\raisebox{-0.4\height}{\begin{overpic}{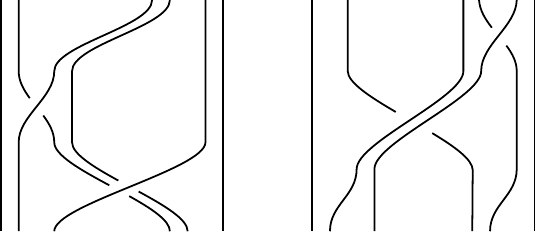} 
\put(50, 21){\makebox[0pt]{$=$}} 
\put(0, -3){\makebox[0pt]{\scaleobj{0.8}{X_1}}}
\put(4, -3){\makebox[0pt]{\scaleobj{0.8}{X_2}}}
\put(10.5, -3){\makebox[0pt]{\scaleobj{0.8}{X_3}}}

\put(32, -3){\makebox[0pt]{\scaleobj{0.8}{Y_1}}}
\put(36, -3){\makebox[0pt]{\scaleobj{0.8}{Y_2}}}
\put(42.5, -3){\makebox[0pt]{\scaleobj{0.8}{Y_3}}}

\put(58, -3){\makebox[0pt]{\scaleobj{0.8}{X_1}}}
\put(62, -3){\makebox[0pt]{\scaleobj{0.8}{X_2}}}
\put(70, -3){\makebox[0pt]{\scaleobj{0.8}{X_3}}}

\put(88.5, -3){\makebox[0pt]{\scaleobj{0.8}{Y_1}}}
\put(92, -3){\makebox[0pt]{\scaleobj{0.8}{Y_2}}}
\put(100, -3){\makebox[0pt]{\scaleobj{0.8}{Y_3}}}

\put(50, 101){.}
\end{overpic}} \end{equation}
\vspace*{3mm}

These tangles are easily seen to be isotopic. The compatibility with units is obvious, since the associator acts trivially on the monoidal unit of $\Umod{\g^{\oplus 3}}$.
\end{proof}
If we additionally perform reduction at the marked point $+$, the two functors $T_{(3)}$, $T'_{(3)}$ above will coincide. 
\begin{proposition}\label{prop:two-fusions-and-reduction}
Post-composing the functors $T_{(3)}, T'_{(3)} \colon \Umod{\g^{\oplus 3}} \to \Umodnp{\g}$ with the reduction $\g\text{-inv} \colon \Umodnp{\g} \to \mathrm{Vect}$ results in identical monoidal functors \[\g\text{-inv} \circ T_{(3)} = \g\text{-inv}\circ T'_{(3)}.\]
\end{proposition}

\begin{proof}
	Applying the reduction at the associator natural transformation (i.e.\ whiskering) tells us that $\g\text{-inv} \circ T_{(3)}$, $\g\text{-inv}\circ T'_{(3)}$ are isomorphic via a monoidal natural transformation given by $\Phi$ acting on the $G$-invariants in $X$, with respect to the diagonal action. In other words, for $x\in X^\g$ and $e\in \g$, we have
	\[ (e)_{X_3} (x) = - e_{X_1}(x) - e_{X_2}(x). \] We can thus write the action of $\Phi\in \mathcal P_3$ only via the first two actions, by doubling the last, reversing orientation, adding minus sign for each chord endpoint: this is exactly the operation $r_3$ from Proposition \ref{prop:operation-r-on-chord-diagrams}. However, since $r_3$ is an algebra morphism and it sends $t_{23}$ to $-t_{12}-t_{22}$, applying it on $\Phi = \exp(\text{Lie series}(t_{12}, t_{23}))$ gives the identity (recall that $t_{ii}$ are central). This fact, that associators act trivially on invariants, was proven by Ohtsuki in \cite[Eq.~10.1]{Ohtsuki2001} and, for the KZ associator, by Le--Murakami in \cite[Lemma~2]{LMMO}.
\end{proof}

In total, we have six possibilities for a star product on $G^3$ seen as the moduli space of flat connections on a disk with three points marked $-$. By the results above, these are all equal, and will be denoted by the disk in the middle: 
\[
\begin{overpic}{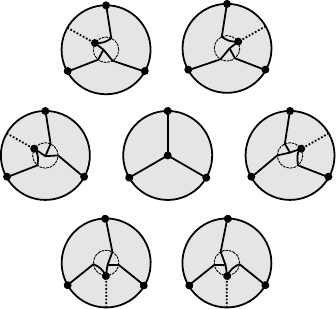}
\put(29, 91){$-$}
\put(19, 72){$-$}
\put(39, 72){$-$}
\put(25, 82){$+$}

\put(65, 91){$-$}
\put(55, 72){$-$}
\put(74, 72){$-$}
\put(68, 82){$+$}

\put(11, 59.5){$-$}
\put(0, 41){$-$}
\put(20, 41){$-$}
\put(7, 50){$+$}

\put(84, 59.5){$-$}
\put(73.5, 41.5){$-$}
\put(94, 41.5){$-$}
\put(88, 50){$+$}

\put(28.5, 27.5){$-$}
\put(19, 10){$-$}
\put(39, 10){$-$}
\put(26, 6){$+$}

\put(65, 27.5){$-$}
\put(55, 10){$-$}
\put(74, 10){$-$}
\put(68, 6){$+$}

\put(47, 59.5){$-$}
\put(30, 41){$-$}
\put(55, 41){$-$}
\put(47, 41){$+$}

\end{overpic}
\]
For example, the top (and then every second) pair of disks give the same quantization, by Proposition \ref{prop:reduction-and-opposite-fusion}.  The bottom (and every second) pair of disks differ by parenthesization, their equality follows from Proposition \ref{prop:two-fusions-and-reduction}. We can denote all these quantizations by the disk in the middle.

\paragraph{Quantizations of the same surface which are not equal.}
Let us mention that the following two moduli spaces 
\begin{equation}\label{pic:plus-minus-two-ways} \begin{overpic}{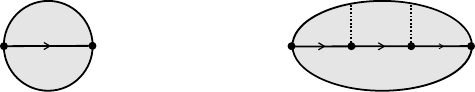}
\put(-4, 8.5){$-$}
\put(21, 8.5){$+$}
\put(55.5, 8.5){$-$}
\put(72.5, 5.5){$+$}
\put(85, 5.5){$-$}
\put(101, 8.5){$+$}
\end{overpic}  \end{equation}
don't have the same quantizations. Indeed the quantization assigned to the left disk is the undeformed product on $C^\infty(G,\mathbb{C})$, while the quantization on the right is given by the product of $K$ and $K_{-t}$ from the proof of Proposition \ref{prop:reduction-and-opposite-fusion}, i.e.\
\[ a e^{t/2} b e^{-t/2} = ab \stackrel{\eqref{eq:product-of-ab}}{=} \Delta \nu \cdot \nu^{-1} \otimes \nu^{-1}.\]
In other words, the action of $\nu$ (by left- or right-invariant vector fields) is an isomorphism between these star products.
\smallskip

A similar example for the disk with three marked points is the following
\[ \begin{overpic}{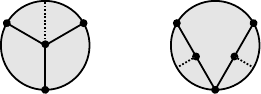}
\put(14, -4){$-$}
\put(-4, 28){$-$}
\put(32.5, 28){$-$}
\put(20, 15){$+$}

\put(79.5, -4){$-$}
\put(61.5, 28){$-$}
\put(98.5, 28){$-$}
\put(75, 15){$+$}
\put(84, 15){$+$}
\end{overpic}\]
We studied the quantization of the left disk above, it's given by the action of the tangle \eqref{eq:JJ1fusion} on diagonal invariants. The quantization of the right disk is instead given by fusing together the disk with two points marked $-$, which has a star product given by $K= be^{-t/2}$
\[ \raisebox{-0.3\height}{\begin{overpic}{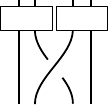} 
\put(1, 73){\footnotesize $be^{-t/2}$}
\put(53, 73){\footnotesize $b e^{-t/2}$}
\put(90, -2){\footnotesize $-t$}
\end{overpic}}\quad . \]
The $-t$ at the bottom is there because we fuse two points marked $-$. See Corollary \ref{cor:LBS-isomorphisms} for an explicit isomorphism in this case.

\subsection{Twisted fusion} \label{sec:twisted-fusion}
We can informally characterize the above counterexamples by saying that changes of spines involving vertices of opposite signs introduce non-trivial isomorphisms between quantizations. We will now twist the fusion functor to make these isomorphisms trivial. Then, we will reprove Propositions \ref{prop:reduction-and-opposite-fusion}, \ref{prop:two-fusions}  and \ref{prop:two-fusions-and-reduction} for the twisted fusion.

\begin{definition}
The twisted fusion is equal to the (untwisted) fusion from Definition \ref{def:fusionreduction} as a functor
\[ T^\text{tw} \colon \Umod{\g\oplus \g\oplus \h}\to \Umod{\g\oplus \h} \]
but with the monoidal structure given by the action
\[  J^\text{tw} \coloneq (\Delta_1 \Delta_2 a^{\tfrac 12})^{1324} \cdot J \cdot  a^{-\tfrac 12}\otimes  a^{-\tfrac 12} \in \mathcal P_4,\]
which we represent by the mixed tangle--chord diagram 
\[
\raisebox{-0.45\height}{\begin{overpic}{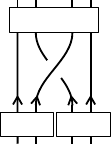}
\put(3, 8){\footnotesize{$a^{-1/2}$}}
\put(42, 8){\footnotesize{$a^{-1/2}$}}
\put(26, 81){\footnotesize{$a^{1/2}$}}
\end{overpic}}.
\]  
\end{definition}
\begin{proposition} \label{prop:twisted-untwisted-isomorphic}
The twisted fusion is a monoidal functor monoidally isomorphic to the untwisted fusion. The twisted fusion of algebras results in quasi-Poisson fusion on the level of the semiclassical quasi-Poisson bracket.
\end{proposition}
\begin{proof}
The action of $a^{\tfrac 12}$ is a natural automorphism of the functor $T$, and hence it can be used to twist a monoidal structure on $T$.
\smallskip

The second statement follows from the fact that $a$ is equal to $1$ plus terms with two and more chords (Proposition \ref{prop:properties-of-a-and-b}).
\end{proof}

Let us now prove a twisted version of Propositions \ref{prop:reduction-and-opposite-fusion}, \ref{prop:two-fusions} and \ref{prop:two-fusions-and-reduction}.
\begin{proposition}\label{prop:twisted-properties}
The two composed functors 
\[\begin{tikzcd}
	{\Umod{\g^{\oplus2}}} & {\Umodnp{\g}} & {\mathrm{Vect}}
	\arrow["T^\text{tw}", curve={height=-18pt}, from=1-1, to=1-2]
	\arrow["{T^\text{tw,op}}"', curve={height=18pt}, from=1-1, to=1-2]
	\arrow["{\g\text{-inv}}", from=1-2, to=1-3]
\end{tikzcd}\]
agree as monoidal functors.
\smallskip

The two possible bracketings of triple fusion $\Umod{\g^{\oplus 3}} \to \Umodnp{\g}$ are isomorphic via the twisted associator $\Phi(t_{12}, t_{23})^{\sqrt{a}}\in \mathcal P_3$. After composing with the reduction, these two functors are equal as monoidal functors.
\end{proposition}
\begin{proof}
The first statement is proven as follows. Together with the twisting by $\sqrt{a}$, the monoidal structure on the functor $\g\text{-inv}\circ T^\text{tw}$ is given by
\[  
\begin{overpic}{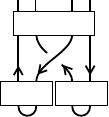}
\put(3, 14){\footnotesize{$a^{-1/2}$}}
\put(52, 14){\footnotesize{$a^{-1/2}$}}
\put(35, 74){\footnotesize{$a^{1/2}$}}
\end{overpic}
\]
On the bottom, the effect of reversing the strand and connecting with a cap is that of the operation $r_2$, i.e.\ we can replace $a^{-\tfrac 12}$ with $r_2(a^{-\tfrac 12})\stackrel{\eqref{eq:ri-of-ab}}{=} \nu^{\tfrac 12}$.

The chord diagram coming from the part below the top box is thus \[ \nu^{\tfrac 12} \otimes \nu^{\tfrac 12} \cdot b \cdot e^{-t/2}.\] We can slide it to the top, through one of the doubled legs of $a^{\tfrac 12}$ by Proposition \ref{prop:insert}. 
We can thus rewrite the tangle above as follows, the bottom greyed-out part carries no further chords
\[  
\raisebox{-0.4\height}{\begin{overpic}{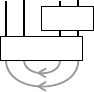}
\put(60, 70){\footnotesize{$X$}}
\put(31, 38){\footnotesize{$a^{1/2}$}}
\end{overpic}} \quad \text{where} \quad X = \nu^{\tfrac 12} \otimes \nu^{\tfrac 12} \cdot b \cdot e^{-t/2}
\]
What remains is the doubling of $r_2(a^{\tfrac 12})$, i.e.\ $\Delta \nu^{-\tfrac 12}$. Gathering all the terms, the star product on a disk with two marked $-$ points is given by the action of 
\begin{equation}\label{eq:K-twisted} K^{\text{tw}} = b \cdot e^{-t/2} \cdot \nu^{\tfrac 12} \otimes \nu^{\tfrac 12} \cdot \Delta \nu^{-\tfrac 12}. \end{equation}
An analogous computation in the case of the opposite fusion gives the same result, using Proposition \ref{prop:S-nu}.
\medskip

Let us now prove the second statement. Using the twisted fusion for left parenthesization results in the following chord diagram. We leave the doubling operations $\Delta_i$ implicit.
\begin{equation} \label{diag:double-fusion-twisted}
\raisebox{-0.45\height}{\begin{overpic}{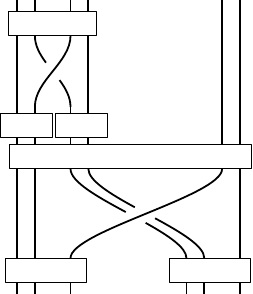}
\put(7, 6){\footnotesize{$a^{-1/2}$}}
\put(65, 6){\footnotesize{$a^{-1/2}$}}
\put(40, 44){\footnotesize{$a^{1/2}$}}

\put(1, 55){\footnotesize{$a^{-1/2}$}}
\put(20, 55){\footnotesize{$a^{-1/2}$}}
\put(14, 90){\footnotesize{$a^{1/2}$}}
\end{overpic}}
\end{equation}
We postcompose with (doubling) of the twisted associator
\[  
\raisebox{-0.45\height}{\begin{overpic}{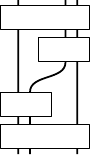}
\put(20, 7){\footnotesize{$a^{-1/2}$}}
\put(3, 28){\footnotesize{$a^{-1/2}$}}
\put(32, 63){\footnotesize{$a^{1/2}$}}
\put(20, 84){\footnotesize{$a^{1/2}$}}
\end{overpic}}.
 \]
The other side of the equation is the monoidal structure on the other parenthesization precomposed with two copies of the twisted associator.
\[\begin{overpic}{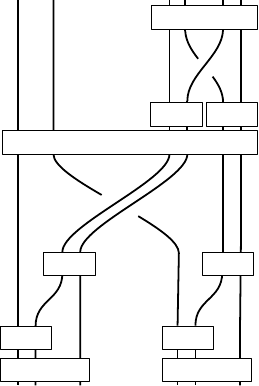}
\put(8, 2){\footnotesize{$a^{-1/2}$}}
\put(48, 2){\footnotesize{$a^{-1/2}$}}
\put(1, 10){\footnotesize{$a^{-1/2}$}}
\put(43, 10){\footnotesize{$a^{-1/2}$}}
\put(14, 29.5){\footnotesize{$a^{1/2}$}}
\put(55, 29.5){\footnotesize{$a^{1/2}$}}

\put(40, 68){\footnotesize{$a^{-1/2}$}}
\put(54, 68){\footnotesize{$a^{-1/2}$}}
\put(32, 61){\footnotesize{$a^{1/2}$}}
\put(48, 93.5){\footnotesize{$a^{1/2}$}}

\end{overpic}
\]
It is a pleasant exercise to prove the desired equality, sliding the rectangles around using Proposition \ref{prop:insert}.

\medskip

To prove the third claim, that the twisted associator $\Phi^{\sqrt{a}}$ acts trivially on invariants, we can't use the argument from Proposition \ref{prop:two-fusions-and-reduction}, since this twisted associator is not necessarily horizontal (and thus doesn't have to come from a Lie series). Instead, we use the algebra morphism $r_3 \colon \mathcal P_3 \to \mathcal P_2$ on the associator twisted with $a$
\[ r_3(\Phi^a) = \underbrace{r_3(\Delta_2 a \cdot (1 \otimes a))}_{x^2\coloneq} \cdot \underbrace{r_3(\Phi)}_{1} \underbrace{r_3(\Delta_1 a^{-1} \cdot (a^{-1}\otimes 1))}_{y^2\coloneq} . \]
However, by Proposition \ref{prop:twist-by-phi}, $\Phi^{a}= \Phi_{-t}$, which is again horizontal, and thus Proposition \ref{prop:two-fusions-and-reduction} applies to it. We get
\[x^2 y^2 = 1 \in \mathcal P_2\]
This implies that $x^2$ and $y^2$ commute, and therefore so do their square roots (computed as power series), and 
\[ \sqrt{x^2 y^2} = xy = 1. \]
Coming back to $\Phi^{\sqrt a}$, we get\footnote{There is also an alternative proof in \cite[Prop.~4.1.6]{Jan:thesis}, from which it follows that $x^2 =a\cdot \nu^{-1}$.}
\[r_3(\Phi^{\sqrt a}) = x r_3(\Phi) y = xy = 1.\]
\end{proof}
\begin{remark}\label{rem:twisted-horizontal}
	We remark that $\Phi^{\sqrt{a}}$ is even, by Propositions \ref{prop:twist-by-phi} and \ref{prop:properties-of-a-and-b}. Since $r_3(\Phi^{\sqrt{a}}) = 1$, it is natural to conjecture that it is moreover a horizontal associator.
\end{remark}

\subsection{Deformation quantization and uni-trivalent spines}
Let us summarize what we know. For both untwisted and twisted fusion, we can specify a way to glue a surface from disks, and thus a deformation quantization of the moduli space, using the following kind of spine.  
\begin{definition}\label{def:LBS-spines}
Let $\tilde{\Gamma}$ be a spine of the marked surface $(\Sigma, V = V_+ \sqcup V_-)$ and, additionally, we
\begin{itemize}
\item color all vertices with $+$ and $-$ in such a way that each edge connects $+$ and $-$,  
\item resolve each external vertex of valency $\ge 2$ with a binary planar tree, rooted at boundary, as in \eqref{diag:two-ways-fourvertex},
\item resolve each internal vertex of valency $\ge 3$ with a trivalent planar rootless tree.
\end{itemize}
We call such spine a Li-Bland--Ševera spine. We note that the signs of the internal vertices are fixed by the signs of the external vertices. 
\end{definition}
\begin{proposition}
	A Li-Bland--Ševera spine uniquely determines a deformation quantization of $(\Sigma, V)$ into an algebra in $\Umod{\g^{V_+} \oplus \bar{\g}^{V_-}}$, when using either the untwisted or the twisted fusion.
\end{proposition}
\begin{proof}
	The decomposition of the surface tells us how to build it from the building block: the disk with two points, labelled $+$ and $-$.
	At the boundary, the binary planar tree determines the order and parenthesization of the fusions.\footnote{It follows moreover from a mild generalization of Proposition \ref{prop:two-fusions}, see also \cite[Thm.~2.1]{LBS:coherent}, that the different parenthesizations are related by the action of associators.}
	\medskip
	
	In the interior, the Li-Bland--Ševera spine does not specify from which boundary component the vertex came, but all choices give identical star products, by Proposition \ref{prop:reduction-and-opposite-fusion} applied to the last fusion. Similarly, the Li-Bland--Ševera spine does not specify the position of boundary point on the trivalent graph; the star product is again independent of this choice, this time by Proposition \ref{prop:two-fusions-and-reduction}.  
	
	For twisted fusion, this follows from Proposition \ref{prop:twisted-properties}.
\end{proof}
For the untwisted fusion, we gave two examples of genuinely different star products on the same moduli space at the end of Section \ref{sec:case-study}. The next theorem tells us that twisted fusion solves this issue, and that these two examples were the only obstructions.

\begin{theorem}\label{thm:unique-quantizations}
	Let $(\Sigma, V)$ be a marked surface. 
	Choose a deformation quantization of the moduli space $M_{\Sigma, V}(G)$ by choosing a Li-Bland--Ševera spine $\tilde{\Gamma}$ as in Definition \ref{def:LBS-spines}. Then, if we use the twisted fusion, the star product on $M_{\Sigma, V}(G)$ depends only on the uni-trivalent spine $\Gamma$ obtained from $\tilde{\Gamma}$ by 
	\begin{itemize}
		\item forgetting all interior signs
		\item replacing all vertices with the trees resolving them (i.e.\ forgetting the dashed circle on \eqref{diag:two-ways-fourvertex})
		\item forgetting all bivalent vertices.  
	\end{itemize}
\end{theorem}
\begin{proof}
All Li-Bland--Ševera spines giving the same trivalent spine are related by a sequence of the following two moves (and analogous moves with reversed signs)
\begin{enumerate}
	\item Adding/removing two binary vertices
	\begin{equation}
	\label{pic:move-one}
	\raisebox{-0.4\height}{ 
	\begin{overpic}{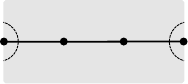}
	\put(3, 35){$+$}
	\put(29.5, 26){$-$}
	\put(61.3, 26){$+$}
	\put(88, 35){$-$}
	\end{overpic}
	}\quad 
	\leftrightarrow
	\quad
	\raisebox{-0.4\height}{ 
	\begin{overpic}{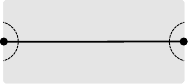}
	\put(3, 35){$+$}
	\put(88, 35){$-$}
	\end{overpic}
	}		
	\end{equation}
	The vertices on the ends are possibly leaves of some resolved vertex, which is denoted by the dashed half-circles.
	\item Moving a leaf inside a resolved vertex to a stand-alone vertex
	\begin{equation}
	\label{pic:move-two}
	\raisebox{-0.4\height}{ 
	\begin{overpic}{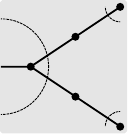}
	\put(3, 90){$+$}
	\put(55, 62.5){$-$}
	\put(55, 30){$-$}
	\put(80, 20){$+$}
	\put(80, 72.5){$+$}
	\put(24, 57){\footnotesize{$2$}}
	\put(24, 35){\footnotesize{$1$}}	
	
	\end{overpic}
	}\quad 
	\leftrightarrow
	\quad
	\raisebox{-0.4\height}{ 
	\begin{overpic}{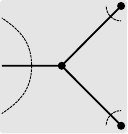}
	\put(3, 90){$+$}
	\put(36, 40){$-$}
		\put(55, 70){\footnotesize{$2$}}
	\put(55, 25){\footnotesize{$1$}}	

	\put(80, 20){$+$}
	\put(80, 72.5){$+$}
	\end{overpic}
	}		
	\end{equation}
	Again, the vertices at the right ends are possibly lying inside a dashed circle.
\end{enumerate}
This is not difficult to see: to upgrade a uni-trivalent spine to a Li-Bland--Ševera spine, one needs to partition the vertices into trees, choose arbitrary signs for the internal trees (not containing any univalent vertices) and then adding bivalent vertices on internal edges (even number between trees with opposite signs and vice versa). The moves above allow to change the partition, the signs of the internal trees (by repeated application of the above moves) and the numbers of bivalent vertices.  

Let us now show that the star products induced on the corresponding moduli spaces $M_{\Sigma, V}(G)$ coincide. The first case, the move \eqref{pic:move-one}, is straightforward. An internal bivalent vertex with $+$ carries an action of $K^\text{tw}$ from \eqref{eq:K-twisted}. The pair of points of opposite signs contributes
\[   K^{\text{tw}} \cdot K^{\text{tw}}_{-t} = a \cdot b \cdot \nu \otimes \nu \cdot \Delta \nu^{-1} \stackrel{\eqref{eq:product-of-ab}}{=} 1. \]
The second case is more involved. Let us look at the right  hand side, where we denote the two actions by $1$ and $2$ as in \eqref{pic:move-two}. At the trivalent vertex with $-$, we have to perform two fusions, i.e.\ the diagram \eqref{diag:double-fusion-twisted}. We are free to choose which strand is the action $1$ and which is $2$, by Proposition \ref{prop:twisted-properties}. 

Since we act on invariants, we can again convert everything into $\g$-actions $1$ and $2$ only, by doubling the third and sixth strand, reversing them and connecting. This connection contributes a doubled cup, if we include it in the Kontsevich integral, we have to cancel the two contributions of $b$ by inserting $b^{-1}$. All together, the star product is given by the action of 
\[ Z \left(  \raisebox{-0.49\height}{
\begin{overpic}{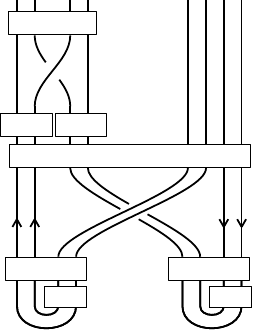}
\put(8, 16){\footnotesize{$a^{-1/2}$}}
\put(58, 16){\footnotesize{$a^{-1/2}$}}
\put(36, 50){\footnotesize{$a^{1/2}$}}

\put(15, 7.5){\footnotesize{$b^{-1}$}}
\put(66, 7.5){\footnotesize{$b^{-1}$}}

\put(54, 103){\footnotesize{$f_1$}}
\put(60, 103){\footnotesize{$f_2$}}
\put(66, 103){\footnotesize{$g_1$}}
\put(72, 103){\footnotesize{$g_2$}}

\put(1, 60){\footnotesize{$a^{-1/2}$}}
\put(18, 60){\footnotesize{$a^{-1/2}$}}
\put(10, 91){\footnotesize{$a^{1/2}$}}
\end{overpic}}  \right)_{-t} 
\; = \; 
Z \left(  \raisebox{-0.49\height}{
\begin{overpic}{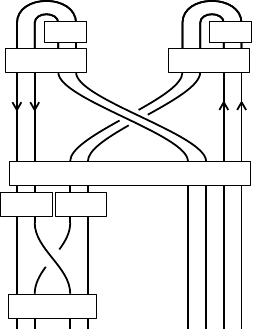}
\put(8, 79){\footnotesize{$b^{-1/2}$}}
\put(58, 79){\footnotesize{$b^{-1/2}$}}
\put(36, 45){\footnotesize{$b^{1/2}$}}

\put(15, 88){\footnotesize{$a^{-1}$}}
\put(66,88){\footnotesize{$a^{-1}$}}

\put(54, -4.5){\footnotesize{$f_1$}}
\put(60, -4.5){\footnotesize{$f_2$}}
\put(66, -4.5){\footnotesize{$g_1$}}
\put(72, -4.5){\footnotesize{$g_2$}}

\put(1, 36){\footnotesize{$b^{-1/2}$}}
\put(18, 36){\footnotesize{$b^{-1/2}$}}
\put(10, 5){\footnotesize{$b^{1/2}$}}
\end{overpic}}  \right) \] 
\vspace{1mm}

To get the equality, we used the compatibility of $Z$ with reflections around horizontal axis \eqref{eq:reflection-about-horizontal-axis} and $a_{-t}=b$ \eqref{eq:ab-minus-t}. Next, we use $r_1(b) = \nu^{-1}$ \eqref{eq:ri-of-ab} and move the boxes around to arrive at
\begin{equation} \label{eq:second-move-final} Z \left(  \raisebox{-0.47\height}{
\begin{overpic}{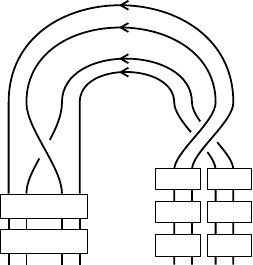}
\put(62, -6){\footnotesize{$f_1$}}
\put(69, -6){\footnotesize{$f_2$}}
\put(78, -6){\footnotesize{$g_1$}}
\put(86, -6){\footnotesize{$g_2$}}

\put(59, 29.5){\footnotesize{$b^{-1/2}$}}
\put(78.5, 29.5){\footnotesize{$b^{-1/2}$}}

\put(62, 4.5){\footnotesize{$a^{-1}$}}
\put(82.5, 4.5){\footnotesize{$a^{-1}$}}

\put(61, 17.1){\footnotesize{$\nu^{1/2}$}}
\put(81, 17.1){\footnotesize{$\nu^{1/2}$}}

\put(9, 20){\footnotesize{$\nu^{-1/2}$}}
\put(10, 6){\footnotesize{$b^{1/2}$}}
\end{overpic}}  \right). \end{equation}
\vspace*{1mm}

Using  $a\cdot b\cdot  \nu\otimes \nu \cdot \Delta \nu^{-1} = 1$ \eqref{eq:product-of-ab}, we can simplify the two chord diagrams on the bottom. On the bottom left, we have $\Delta_1 \Delta_2$ of the following chord diagram
\begin{equation}\label{eq:m2-lhs-left}
\quad b^{1/2} \cdot \Delta\nu^{-1/2} = a^{-1/2} \cdot \nu^{-1/2} \otimes \nu^{-1/2}.
\end{equation}
On the bottom right, we have two copies of 
\begin{equation}\label{eq:m2-lhs-right}
b^{-1/2} \cdot \Delta \nu^{1/2} \cdot a^{-1} = a^{-1/2} \cdot \nu^{1/2} \otimes \nu^{1/2}.
\end{equation}

The other side of the move \eqref{pic:move-two} produces a star product with a single twisted fusion and the terms $K^\text{tw}_{-t}=a\cdot e^{t/2}\cdot \Delta \nu^{-1/2} \cdot \nu^{1/2}\otimes \nu^{1/2}$ from the two bivalent vertices
\[ Z \left(  \raisebox{-0.47\height}{
\begin{overpic}{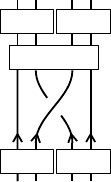}
\put(3, 6.5){\footnotesize{$a^{-1/2}$}}
\put(33, 6.5){\footnotesize{$a^{-1/2}$}}

\put(5,-9){\footnotesize{$f_1$}}
\put(16, -9){\footnotesize{$f_2$}}
\put(35, -9){\footnotesize{$g_1$}}
\put(46, -9){\footnotesize{$g_2$}}

\put(22, 64){\footnotesize{$a^{1/2}$}}

\put(3, 84){\footnotesize{$K^\text{tw}_{-t}$}}
\put(34, 84){\footnotesize{$K^\text{tw}_{-t}$}}
\end{overpic}}  \right). \]
\vspace*{3mm}

We can make it look more similar to \eqref{eq:second-move-final} by adding a four-strand cap (a ``comb-over'') to the left, for the price of cancelling a $a\otimes a \cdot \Delta_1 \Delta_2 a$ term. The factors $e^{t/2}$ can be made into overcrossings, and we get
\begin{equation} \label{eq:move2-rhs}
 Z \left(  \raisebox{-0.47\height}{
\begin{overpic}{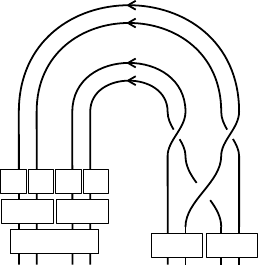}
\put(60, -6){\footnotesize{$f_1$}}
\put(67, -6){\footnotesize{$f_2$}}
\put(81, -6){\footnotesize{$g_1$}}
\put(88.5, -6){\footnotesize{$g_2$}}

\put(59, 4.5){\footnotesize{$a^{-1/2}$}}
\put(80, 4.5){\footnotesize{$a^{-1/2}$}}

\put(0, 29.5){\footnotesize{$\sqrt\nu$}}
\put(10, 29.5){\footnotesize{$\sqrt\nu$}}
\put(20.4, 29.5){\footnotesize{$\sqrt\nu$}}
\put(30.7, 29.5){\footnotesize{$\sqrt\nu$}}

\put(23, 17.5){\footnotesize{$\nu^{-1/2}$}}
\put(1, 17.5){\footnotesize{$\nu^{-1/2}$}}

\put(10, 6){\footnotesize{$a^{-1/2}$}}
\end{overpic}}  \right). \end{equation}
\vspace*{3mm}
Moving the $\nu^{1/2}$ terms to the other side, we get exactly the chord diagrams from \eqref{eq:m2-lhs-left} and \eqref{eq:m2-lhs-right}. It remains to notice that the two tangles above are identical.
\end{proof}
\begin{remark}
The equality of the tangles in \eqref{eq:second-move-final} and \eqref{eq:move2-rhs}  from the above proof is the chord-diagrammatic version of the fact that the fusion functor with undercrossing and overcrossing are naturally monoidally isomorphic, via $e^{t/2}$. 
\end{remark}
Recall that all uni-trivalent spines are connected by a sequence of flips \eqref{pic:flip}. 
\begin{theorem}\label{thm:flips-isomorphisms}
	If a uni-trivalent spine $\Gamma_{e}\subset \Sigma$ is obtained from $\Gamma$ by a flip of an edge $e$, the corresponding deformation quantizations of $C^\infty(M_{\Sigma, V}(G), \mathbb{C})$ are isomorphic by the action of the twisted associator $\Phi^{\sqrt{a}}\colon A_\Gamma \to A_{\Gamma_e}$ at the three edge endpoints $2,3,4$ as below 
	\[\raisebox{-0.47\height}{
\begin{overpic}{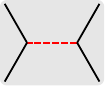}
\put(-6, -13){\footnotesize{$3$}}
\put(92, -13){\footnotesize{$4$}}
\put(-6, 84){\footnotesize{$2$}}
\put(92, 84){\footnotesize{$1$}}
\end{overpic}}.  \]
	If $\Gamma'$ is a different uni-trivalent spine, choose a sequence of flips from $\Gamma$ to $\Gamma'$. Then the isomorphisms $\Psi_{\Gamma \to \Gamma'} \colon A_\Gamma \to A_{\Gamma'}$, given by composing the associators, does not depend on the sequence of flips we have chosen.   
\end{theorem}
\begin{proof}
We can assume that the edge $e$ along which we flip is an internal edge of some tree of the  Li-Bland--Ševera spine. Then the first result follows from Proposition \ref{prop:twisted-properties}. 

For the second claim, we need to verify that the isomorphisms satisfy the three relations from  Theorem \ref{thm:flips}. The second relation is satisfied, since the actions at different half-edges commute, if we e.g.\ make all internal vertices to be distinct in a Li-Bland--Ševera spine. The third relation, the pentagon, holds since twisted associators satisfy the same pentagon equation \eqref{eq:assoc:pentagon}. It remains to check the first relation, which can be written as 
\begin{equation} \label{eq:PhiPhi} (\Phi^{\sqrt{a}})^{123} \cdot (\Phi^{\sqrt{a}})^{234} = 1 \end{equation}
acting on invariants with respect to the diagonal $G$-action. This follows from the pentagon equation \eqref{eq:assoc:pentagon} for the twisted associator, where the terms with $\Delta_i \Phi^{\sqrt{a}}$ disappear by $r_3(\Phi^{\sqrt{a}})=1$ from  Proposition \ref{prop:twisted-properties} (on the level of chord diagrams, this follows from the fact that $r_i$ and $\Delta_j$ commute up to suitable reindexing). We remark that, for an even associator, the twist $\sqrt{a}$ is trivial by \eqref{eq:ab-even-associator}, and \eqref{eq:PhiPhi} reduces to $\Phi \Phi^{321}= 1$, i.e.\ the defining property \eqref{eq:assoc:inverse}. 
\end{proof}
Since the twisted associator $\Phi^{\sqrt{a}}$ is even (see Remark \ref{rem:twisted-horizontal}), the isomorphism does not depend on the sign of the vertex in which it acts. Similarly, by looking at the flipped edge upside down, instead of the action $(\Phi^{\sqrt{a}})^{234}$ we would get $(\Phi^{\sqrt{a}})^{412}$; these are equal applying equation \eqref{eq:PhiPhi} and its renumbering $(4123)$.
\medskip

If there are no marked points on the boundary, we can consider the algebra of $\g^V$-invariant functions on $M_{\Sigma, V}(G)$, which has an associative star product.
\begin{corollary}
	Let $\Sigma$ be a surface with non-empty boundary in each connected component. Then Theorems \ref{thm:unique-quantizations} and \ref{thm:flips-isomorphisms} extend to the moduli space $M_\Sigma(G)$ in the sense that $C^\infty(M_{\Sigma, V}(G), \mathbb C)^{\g^V}$ inherits a star-product which depends only on the trivalent spine $\Gamma\subset \Sigma$, and the maps $\Psi_{\Gamma \to \Gamma'}$ give a collection of algebra isomorphisms closed under composition. 
\end{corollary}
\begin{proof}
	The presentation of the Ptolemy groupoid from Theorem \ref{thm:flips} also works for surfaces without marked points on the boundary. The fact that different choices of $(\Sigma, V)$ give the same star product, which thus depends only on the trivalent graph, follows from the fact that the resulting star product does not depend on the order in which we take the $\g$-invariants.
\end{proof}

\medskip
Let $F$ be a mapping class which fixes each boundary component pointwise. Denote by $F^*$ the automorphism of  $M_{\Sigma, V}(G)$ given by $F^*(\rho)(\gamma) = \rho(F(\gamma))$; the pullback of this automorphism $F\coloneq(F^*)^* \colon C^{\infty}(M_{\Sigma, V}(G), \mathbb C) \to C^{\infty}(M_{\Sigma, V}(G), \mathbb C)$ intertwines the star products on $A_\Gamma$ and $A_{F\Gamma}$ for trivial reasons.
\begin{corollary}
For a mapping class $F$, let $\Psi_{\Gamma \to F\Gamma}\colon A_\Gamma \to A_{F\Gamma}$ be the isomorphism given by a sequence of flips from the uni-trivalent spine $\Gamma$ to $F(\Gamma)$. Then the algebra isomorphisms $\Psi_F \coloneq F^{-1}\circ \Psi_{\Gamma \to F\Gamma}\colon A_\Gamma \to A_\Gamma$ define a left action of the mapping class group on the deformation quantization $A_\Gamma$.
\end{corollary}
\begin{proof}
One needs to check that this assignment is compatible with composition of mapping classes. This follows since the sequence of flips which takes $\Gamma$ to $F'(\Gamma)$ also takes $F(\Gamma)$ to $F'F(\Gamma)$.
\end{proof}
For example, a flip 
\[ \raisebox{-0.45\height}{\begin{overpic}{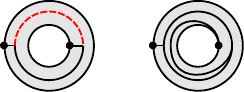}
\put(47, 18){$\leadsto$}
\end{overpic}
} \]
acts by $(\Phi^{\sqrt{a}})(L1, L2, R1)$ on $C^\infty(M_{\mathrm{Ann}, \{v_1, v_2\} }, \mathbb C)(G) \cong C^\infty(G\times G, \mathbb{C})$.\footnote{Superscripts denote left- and right-invariant vector field on the two $G$ factors, which are given by the holonomy along paths going to the left, over and under the inner circle.} Therefore the action of the mapping class twisting the inner boundary of the annulus clockwise is given by such $\Phi$ followed by the diffeomorphism $(g_1, g_2) \mapsto (g_2, \Ad_{g_2^{-1}}g_1)$.
\medskip

Since the twisted and untwisted fusion give isomorphic functors (and therefore deformation quantizations), we can also use the above result to characterize isomorphisms between the original Li-Bland--Ševera \cite{LBS:DQ} quantizations.
\begin{corollary}\label{cor:LBS-isomorphisms}
	There are algebra isomorphisms between each pair of deformation quantizations given by Li-Bland--Ševera spines, and these isomorphisms are closed under composition (as in \eqref{diag:coherent-quantizations}). Explicitly, these isomorphisms are given by 
	\begin{enumerate}
		\item the action of $\nu^{-1}$ for removing a pair of neighboring bivalent vertices,
		\item the action of the  associator $\Phi$ for a reparenthesization (acting as in Theorem \ref{thm:flips-isomorphisms} but the associator is untwisted),
		\item the action of $a^{-1}$ for the move described in \eqref{pic:move-two}.
	\end{enumerate}
\end{corollary} 
\begin{proof}[Sketch of a proof]
Each Li-Bland--Ševera spine $\tilde{\Gamma}$ determines two algebras quantizing the same moduli space, using twisted and untwisted fusion. The untwisted fusion case is the deformation quantization from \cite{LBS:DQ}, while in the twisted case we get the deformation quantization assigned to the uni-trivalent spine $\Gamma$. These are isomorphic by Proposition \ref{prop:twisted-untwisted-isomorphic}, and the isomorphism from the untwisted to the twisted algebra is given by actions of $a^{1/2}$ for each fused pair of half-edges. Composing these isomorphisms with the isomorphisms $\Psi$ relating different uni-trivalent spines gives the system of isomorphisms with the desired composition property.

Let us sketch the three isomorphism from the statement. When removing a pair of bivalent vertices explicitly, we act on invariants via the isomorphism $a^{1/2}$, which simplifies to
\[ r_2(a^{1/2}) \stackrel{\eqref{eq:ri-of-ab}}{=} \nu^{-1/2}. \]
A pair of such vertices contributes $\nu^{-1}$ since $\nu_{-t} = \nu$. In the second item, the factors $a^{1/2}$ cancel the twisting of the associator $\Phi^{\sqrt{a}}$. In the third item, relating the two sides of \eqref{pic:move-two} through a uni-trivalent spine gives
\[( a^{1/2}\cdot \nu^{-1/2}\otimes \nu^{-1/2} )^{-1} \cdot r_3(\Delta_{1} b^{1/2} \cdot b^{1/2} \otimes 1) , \]
which simplifies to $a^{-1}$.
\end{proof}

\subsection{Loose ends}
The basic building block we used in Proposition \ref{prop:building-block} was the non-deformed algebra $C^\infty(G, \mathbb{C}) \in \Umod{\g \oplus \bar{\g}}$. To relate the deformation quantization via factorization homology to the quantization presented here, we note that the role of $\bar{\g}$ can be understood as taking the opposite braiding; compare with \cite[Prop.~7.19]{KKMP}.\footnote{This interpretation of $a$ was suggested to me by P. Safronov.}
\begin{proposition}
	The identity functor 
	\[ \Umodnp{\bar{\g}} \to {\Umodnp{\g}}^\text{br. op} \]
	is a braided monoidal equivalence, with the monoidal structure given by $a$
\end{proposition}
\begin{proof}
Follows from Proposition \ref{prop:twist-by-phi}.
\end{proof}
For quantum groups, an analogous statement is that the bar involution \cite[Sec.~3.11]{Lusztig:QG} reverses braiding \cite[Sec.~8.2]{Gaitsgory:KZ}.

For an even associator, some of the above constructions simplify using \eqref{eq:ab-even-associator}; the twisted fusion, however, still differs from the untwisted one. 
\begin{proposition}
Let $\Phi$ be an even associator. Then the deformation quantization of a disk with two vertices marked $+$ is given by $K^{\text{tw}}_{-t} = e^{t/2}$, and for two vertices marked $-$ it is given by $K^{\text{tw}}=e^{-t/2}$.  The isomorphisms from flips are given by the action of $\Phi$.
\end{proposition}
For an even associator, \cite[Last item below Prop.~3]{LBS:DQ} also note that the undeformed product of $C^\infty(G,\mathbb{C})$ is associative also in $\Umod{\g^{\oplus 2}}$ and $\Umod{\bar{\g}^{\oplus 2}}$. This undeformed product is isomorphic to the one given by $e^{t/2}$ by the action of the square root of the ribbon element $e^{t_{11}/4}$.

\section{Traces on moduli spaces.}
A natural definition of a trace of an algebra in a braided monoidal category should use the braiding or twist to formulate the cyclic property of the trace. In the present case, the correct notion turns out to be the so-called half-twist \cite{SnyderTingley:half-twist}
\begin{equation}\label{eq:half-twist}
	\tau_{A, A} \circ \Delta e^{\hbar t_{11}/4} =  \tau_{A, A} \circ e^{\hbar t_{12}/2} \circ  e^{\hbar t_{11}/4} \otimes  e^{\hbar t_{11}/4}  \colon A \otimes A \to A\otimes A
\end{equation}
Here $\tau_{A, A} \colon A \otimes A \to A \otimes A$ is the map $a\otimes a' \mapsto a' \otimes a$.
\begin{definition}
Let $(A, m)\in \Umodnp{\g}$ be an associative algebra. A trace on $A$ is a morphism $\operatorname{tr}\colon A \to \mathbb C$ (i.e.\ to the monoidal unit) in $\Umodnp{\g}$ which satisfies one of the following equivalent conditions
\begin{align}
\operatorname{tr}\circ m \circ \tau_{A, A} \circ \Delta e^{\hbar t_{11}/4}  &=  \operatorname{tr}\circ m  \label{tr:cyclic1}\\
\operatorname{tr}\circ m \circ \beta_{A, A} &= \operatorname{tr}\circ m \circ e^{-\hbar t_{11}/4} \otimes e^{-\hbar t_{11} /4} \label{tr:cyclic2} \\
\operatorname{tr}\circ m \circ \tau_{A,A}&=  \operatorname{tr}\circ m \label{tr:cyclic3}
\end{align}
\end{definition}
\begin{proposition}
	The three properties \eqref{tr:cyclic1}, \eqref{tr:cyclic2} and \eqref{tr:cyclic3} are equivalent.
\end{proposition}
\begin{proof}
The equivalence of the first two equations comes from 
\[ \Delta{t_{11}} =  t_{11}\otimes 1 + 1\otimes t_{11} + 2 t_{12}.\]
The last equation is equivalent to the first one since $e^{\hbar t_{11}/4}$ is a $\g$-invariant element of $U\g$, i.e.\ it is a natural automorphism of the identity functor on $\Umodnp{\g}$, and thus it commutes with $\operatorname{tr}\circ m \circ \tau_{A, A}$. On the monoidal unit $\mathbb C$ it acts as the identity.
\end{proof}
\begin{example}
Let $G$ be a compact Lie group and $\mu$ a normalized bi-invariant Haar measure. Then 
\[ f \mapsto \int_G f \mu  \]
is a trace on the undeformed algebra $C^{\infty}(G, \mathbb{C})\in \Umod{\g\oplus \bar{\g}}$. We recall that all compact Lie groups are unimodular, i.e.\ they possess a bi-invariant Haar measure.
\end{example}

\begin{proposition} \label{prop:trace-and-fusion}
Let $(A, m)$ be an algebra in $\Umod{\g\oplus \g\oplus \h}$ and let $\operatorname{tr}$ be a trace on $A$. Then the same morphism $\operatorname{tr}\colon T(A) \to \mathbb C$ is a trace on the fusion of $A$. The same holds for the twisted fusion.
\end{proposition}
\begin{proof}Let us start with the untwisted fusion. We want to show that
	\[ \tr \circ m \circ J_A \circ \tau_{A, A}  = \tr \circ m \circ J_A. \]
	Exchanging the permutation and $J_A$ on the LHS, we get
	\[\tr \circ m \circ J_A \circ \tau_{A, A} = \tr \circ m \circ \tau_{A, A} \circ (J^{3, 4, 1,2})_A  = \tr \circ m \circ (J^{3,4,1,2})_A\,. \]
	The second equality comes from the fact that $\tr$ is a trace. Thus, we need to prove that\footnote{This is not true without the trace and the multiplication, as $J = 1 + \hbar t_{23} + \dots$, while $J^{3,4,1,2} = 1 + \hbar t_{14} + \dots$.}
	\[\tr \circ m \circ (J^{3,4,1,2})_A =  \tr \circ m \circ J_A.\]
	Since $\tr$ is a morphism from a $\g\oplus \g$-representation $A$ to the trivial representation, it satisfies $\tr \circ (e, e') = 0$, where $(e,e')$ denotes the action of $e, e'\in \g$. 
	The product $m \colon A\otimes A \to A$ is a $U(\g \oplus \g)$-map, i.e. 
	\[ (e, e') \circ m = m \circ (e, e', e, e') .\] 
	Thus, we can move the actions between the first and third slot as
	\[ \tr \circ m \circ (e, 0, 0, 0) =  \underbrace{\tr \circ (e, 0) \circ m}_{=0}  - \tr \circ m \circ (0, 0, e, 0) =\tr \circ m \circ (0, 0, -e, 0), \] 
	and similarly
	\[ \tr \circ m \circ (0, e, 0, 0)= \tr \circ m \circ (0, 0, 0, -e). \]
	Applying this repeatedly, one obtains, for $w_i \in U\g$,  	
	\[ \tr \circ m \circ (w_1\otimes w_2 \otimes w_3 \otimes w_4) = \tr \circ m \circ (S(w_3)w_1 \otimes S(w_4)w_2 \otimes 1 \otimes 1),  \]
	where $S$ is the antipode of $U\g$.
	Applying this to $J= J_{(1)}\otimes J_{(2)} \otimes J_{(3)} \otimes J_{(4)}$, one gets the action of $L := S(J_{(3)})J_{(1)} \otimes S(J_{(4)})J_{(2)}$ on the first and second ``slot'' of $A\otimes A$. For $J^{3, 4, 1, 2}= J_{(3)}\otimes J_{(4)} \otimes J_{(1)} \otimes J_{(2)}$, one obtains the action of $S(J_{(1)})J_{(3)} \otimes S(J_{(2)})J_{(4)}$ at the same slots. Thus, we need that $L \in U\g \otimes U\g$ satisfies $L = S\otimes S(L)$. 
	
	Fortunately, we can compute $L$ easily using the Kontsevich integral: it comes from the following chord diagram
	\[
	Z \left( \raisebox{-0.45\height}{\begin{overpic}{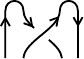} 
	\end{overpic}
	} \right)= Z \left( \raisebox{-0.42\height}{\begin{overpic}{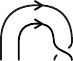} 
	\end{overpic}
	}  \right) = R^{-1}a,
	\]
	where $R^{-1}$ is for the undercrossing and $a$ for the doubled cup. Both of these elements are invariant under $S_{1}S_{2}$ \eqref{eq:Sab}, which is the chord-diagrammatic analogue of $S\otimes S$ . Finally,  $R^{-1}$ is central and we can finish the proof in the untwisted case
	\[ S_{1}S_{2}(L) = S_{1}S_{2}(R^{{-1}}a) = aR^{-1}  = L.\]
	\smallskip

	For the case of the twisted fusion, the term $a^{1/2}$ on the top vanishes due to Proposition \ref{prop:action-on-invariants}, while the two $a^{-1/2}$ combine to $a^{-1}$, i.e.\ the twisted version of the chord diagram $L$ is equal to $L^\text{tw}= R^{-1}$.
	
\end{proof}
Explicitly, applying Proposition \ref{prop:trace-and-fusion} repeatedly tells us that the morphism $\tr \circ m$ can be simplified to the action of a cabling of $a$ according to the number and parenthesization of half-edges at a vertex, followed by $e^{\hbar \sum_{i<j}t_{ij}/2}$, the braiding which reverses the order of strands. In the case of twisted fusion, only the braiding remains. 
\smallskip

\begin{proposition}
	Let $(A, m)$ be an algebra in $\Umod{\g\oplus \h}$ and let $\operatorname{tr}$ be a trace on $A$. Then the same morphism is a trace on the algebra $A^\g$ in $\Umodnp{\h}$.
	\end{proposition}
	\begin{proof}
	Follows immediately from the from the cyclic property \eqref{tr:cyclic3}: if it holds on $A$, it also holds when restricted to $A^\g$.
	\end{proof}

\begin{proposition}
	The isomorphisms $\Psi_{\Gamma \to \Gamma'}$ from Theorem \ref{thm:flips-isomorphisms} 
	intertwine traces. The same is true for the isomorphisms between the untwisted quantizations described in Corollary \ref{cor:LBS-isomorphisms}.
\end{proposition}	
\begin{proof}
	Since the isomorphisms between star products always act by chord diagrams of the form \[1+\text{terms with chords},\] the composition $\operatorname{tr}\circ \Psi$ is equal to $\operatorname{tr}\circ \Psi_{\hbar = 0}$. Thus, we only need to check that the trace is compatible with the isomorphisms $G^{\text{edges}{(\Gamma)}}/G^{\text{internal vertices}(\Gamma)} \cong M_{\Sigma, V}(G)$, i.e.\ that the Haar measures induce the same trace on $C^{\infty}(M_{\Sigma, V}(G), \mathbb C)$, independently of $\Gamma$. 
	
Let us show that the trace doesn't depend on the Li-Bland--Ševera spine. We can again check the two moves from Theorem \ref{thm:unique-quantizations}; which change the number of internal edges and vertices. For example, adding two bivalent vertices increases the number of $G$-factors by $2$; the integration over these factors gives $1$ thanks to the normalization of the Haar measure. A similar argument holds for the move \eqref{pic:move-two}.
\end{proof}
To summarize, we get a well-defined trace on the deformation quantization of the moduli space of flat connections, in the sense that it's independent of the Li-Bland--Ševera or uni-trivalent spine.
We expect these traces to be closely related to the traces of Roche and Szenes \cite{RocheSzenes02}.

\printbibliography

@Article{Goldman86,
  author   = {Goldman, William M.},
  journal  = {Inventiones mathematicae},
  title    = {Invariant functions on {Lie} groups and {Hamiltonian} flows of surface group representations},
  year     = {1986},
  issn     = {1432-1297},
  month    = jun,
  number   = {2},
  pages    = {263--302},
  volume   = {85},
  doi      = {10.1007/BF01389091},
}

@misc{Gaiotto:Kq-algebras,
      title={$K_{\mathfrak{q}}$-algebras}, 
      author={Davide Gaiotto},
      year={2026},
      eprint={2609.38702},
      archivePrefix={arXiv},
      primaryClass={math.QA},
      url={https://arxiv.org/abs/2609.38702}, 
}

@article{SnyderTingley:half-twist,
  title = {The half-twist for
                    <i>U</i>
                    <sub>
                      <i>q</i>
                    </sub>
                    (g) representations},
  volume = {3},
  ISSN = {1937-0652},
  url = {http://dx.doi.org/10.2140/ant.2009.3.809},
  DOI = {10.2140/ant.2009.3.809},
  number = {7},
  journal = {Algebra \& Number Theory},
  publisher = {Mathematical Sciences Publishers},
  author = {Snyder,  Noah and Tingley,  Peter},
  year = {2009},
  month = Nov,
  pages = {809–834},
      eprint={0810.0084},
      archivePrefix={arXiv},
      primaryClass={math.QA},
}

@incollection{BarNatan:braids,
  booktitle = {The Interface of Knots and Physics},
  title = {Vassiliev and Quantum Invariants of Braids},
  author = {Bar-Natan, Dror},
  ISBN = {9780821892664},
  ISSN = {0160-7634},
  DOI = {10.1090/psapm/051},
  journal = {Proceedings of Symposia in Applied Mathematics},
  publisher = {American Mathematical Society},
  pages = {129-144}, 
  volume = {51},
  year = {1996}
}

@Article{EnriquezEtingof,
  author        = {Enriquez, Benjamin and Etingof, Pavel},
  journal       = {American Mathematical Society Translations: Series 2},
  title         = {Quantization of Alekseev-Meinrenken dynamical $r$-matrices},
  year          = {2003},
  issn          = {2472-3193},
  month         = {Sep},
  pages         = {81–98},
  archiveprefix = {arXiv},
  doi           = {10.1090/trans2/210/06},
  eprint        = {math/0302067},
  isbn          = {9781470434212},
  primaryclass  = {math.QA},
  publisher     = {American Mathematical Society},
}

@phdthesis{Jan:thesis,
  doi = {10.13097/ARCHIVE-OUVERTE/UNIGE:154218},
  author = {Pulmann,  Ján},
  title = {On Quantization of Moduli Spaces and Poisson-Hopf Algebras},
  publisher = {Université de Genève},
  year = {2021},
  copyright = {Free access}
}

@article{KKMP,
  title = {Deformation quantization via categorical factorization homology},
  volume = {484},
  ISSN = {0001-8708},
  url = {http://dx.doi.org/10.1016/j.aim.2025.110708},
  DOI = {10.1016/j.aim.2025.110708},
  journal = {Advances in Mathematics},
  publisher = {Elsevier BV},
  author = {Karlsson,  Eilind and Keller,  Corina and M\"{u}ller,  Lukas and Pulmann,  Ján},
  year = {2026},
  month = Jan,
  pages = {110708},
      eprint={2410.12516},
      archivePrefix={arXiv},
      primaryClass={math.QA},
}

@Article{LBS:coherent,
  author        = {Li-Bland, David and \v{S}evera, Pavol},
  journal       = {Mathematical Research Letters},
  title         = {Coherent quantization using coloured surfaces},
  year          = {2019},
  issn          = {1945-001X},
  number        = {3},
  pages         = {805–850},
  volume        = {26},
  archiveprefix = {arXiv},
  doi           = {10.4310/mrl.2019.v26.n3.a7},
  eprint        = {1510.05615},
  primaryclass  = {math.SG},
  publisher     = {International Press of Boston},
}

@Book{CDM,
  author    = {Chmutov, Sergei and Duzhin, Sergei and Mostovoy, Jacob},
  publisher = {Cambridge University Press},
  title     = {Introduction to {Vassiliev} {Knot} {Invariants}},
  year      = {2012},
  isbn      = {9781107020832},
  month     = jul,
}

@Article{LeMurakami:parallel,
  author  = {Le, Thang T. Q. and Murakami, Jun},
  journal = {Journal of Pure and Applied Algebra},
  title   = {Parallel version of the universal {Vassiliev}-{Kontsevich} invariant},
  year    = {1997},
  issn    = {0022-4049},
  month   = oct,
  number  = {3},
  pages   = {271--291},
  volume  = {121},
  doi     = {10.1016/S0022-4049(96)00054-0},
}

@Article{BarNatan:vassiliev,
  author  = {Bar-Natan, Dror},
  journal = {Topology},
  title   = {On the {Vassiliev} knot invariants},
  year    = {1995},
  issn    = {0040-9383},
  month   = apr,
  number  = {2},
  pages   = {423--472},
  volume  = {34},
  doi     = {10.1016/0040-9383(95)93237-2},
  url     = {https://www.math.toronto.edu/~drorbn/papers/OnVassiliev/},
}

@Article{DrinfeldGal,
  author   = {Drinfeld, Vladimir G.},
  journal  = {Algebra i Analiz},
  title    = {On quasitriangular quasi-{Hopf} algebras and on a group that is closely connected with {$\mathrm{Gal}(\overline{\mathbb {Q}}/\mathbb{Q})$}},
  year     = {1990},
  issn     = {0234-0852},
  note     = {(\textit{translation in} Leningrad Mathematical Journal, 1991, 2:4, 829–860)},
  number   = {4},
  pages    = {149--181},
  volume   = {2},
  mrnumber = {1080203},
}

@Article{BarNatanOnAssociators,
  author   = {Bar-Natan, Dror},
  journal  = {Selecta Mathematica},
  title    = {On associators and the {Grothendieck}-{Teichmuller} group, {I}},
  year     = {1998},
  issn     = {1420-9020},
  month    = jun,
  number   = {2},
  pages    = {183},
  volume   = {4},
  doi      = {10.1007/s000290050029},
}

@Article{LeMurakami:universal,
  author  = {Le, Thang T. Q. and Murakami, Jun},
  journal = {Compositio Mathematica},
  title   = {The universal {Vassiliev}-{Kontsevich} invariant for framed oriented links},
  year    = {1996},
  number  = {1},
  pages   = {41--64},
  volume  = {102},
}

@Book{Ohtsuki2001,
  author    = {Ohtsuki, Tomotada},
  publisher = {WORLD SCIENTIFIC},
  title     = {Quantum Invariants},
  year      = {2001},
  doi       = {10.1142/4746},
}

@article{HabiroMassuyeau:handlebodies,
  title = {The Kontsevich integral for bottom tangles in handlebodies},
  volume = {12},
  ISSN = {1664-073X},
  DOI = {10.4171/qt/155},
  number = {4},
  journal = {Quantum Topology},
  publisher = {European Mathematical Society - EMS - Publishing House GmbH},
  author = {Habiro,  Kazuo and Massuyeau,  Gwénaël},
  year = {2021},
  month = Sept,
  pages = {593–703},
      eprint={1702.00830},
      archivePrefix={arXiv},
      primaryClass={math.GT},
}

@article{Gaitsgory:KZ,
  title = {A Conjectural Extension of the Kazhdan–Lusztig Equivalence},
  volume = {57},
  ISSN = {1663-4926},
  DOI = {10.4171/prims/57-3-14},
  number = {3},
  journal = {Publications of the Research Institute for Mathematical Sciences},
  publisher = {European Mathematical Society - EMS - Publishing House GmbH},
  author = {Gaitsgory,  Dennis},
  year = {2021},
  month = Oct,
  pages = {1227–1376},
      eprint={1810.09054},
      archivePrefix={arXiv},
      primaryClass={math.RT},
}

@book{Lusztig:QG,
  title = {Introduction to Quantum Groups},
  ISBN = {9780817647179},
  url = {http://dx.doi.org/10.1007/978-0-8176-4717-9},
  DOI = {10.1007/978-0-8176-4717-9},
  publisher = {Birkh\"{a}user Boston},
  author = {Lusztig,  George},
  year = {2010}
}

@article{Dancso:trivalent,
  title = {On the Kontsevich integral for knotted trivalent graphs},
  volume = {10},
  ISSN = {1472-2747},
  DOI = {10.2140/agt.2010.10.1317},
  number = {3},
  journal = {Algebraic \& Geometric Topology},
  publisher = {Mathematical Sciences Publishers},
  author = {Dancso,  Zsuzsanna},
  year = {2010},
  month = June,
  pages = {1317–1365},
      eprint={0811.4615},
      archivePrefix={arXiv},
      primaryClass={math.GT},
}

@Article{Andersen1998,
  author    = {Andersen, Jørgen E. and Mattes, Josef and Reshetikhin, Nicolai},
  journal   = {Mathematical Proceedings of the Cambridge Philosophical Society},
  title     = {Quantization of the algebra of chord diagrams},
  year      = {1998},
  issn      = {1469-8064, 0305-0041},
  month     = nov,
  number    = {3},
  pages     = {451--467},
  volume    = {124},
  doi       = {10.1017/S0305004198002813},
  publisher = {Cambridge University Press},
}

@Article{Drinfeld:quasi-hopf-algebras,
  author   = {Drinfeld, Vladimir G.},
  journal  = {Algebra i Analiz},
  title    = {Quasi-{Hopf} algebras},
  year     = {1989},
  issn     = {0234-0852},
  note     = {(\textit{translation in} Leningrad Mathematical Journal, 1990, 1:6, 1419–1457)},
  number   = {6},
  pages    = {114--148},
  volume   = {1},
  mrnumber = {1047964},
  url      = {https://mathscinet.ams.org/mathscinet-getitem?mr=1047964},
}

@Article{LMMO,
  author    = {Le, Thang T. Q. and Murakami, Hitoshi and Murakami, Jun and Ohtsuki, Tomotada},
  journal   = {Osaka Journal of Mathematics},
  title     = {A three-manifold invariant via the {Kontsevich} integral},
  year      = {1999},
  issn      = {0030-6126},
  month     = jan,
  number    = {2},
  pages     = {365--395},
  volume    = {36},
  publisher = {Osaka University and Osaka City University, Departments of Mathematics},
}

@Book{Penner:DTT,
  author    = {Penner, Robert C.},
  publisher = {European Mathematical Society},
  title     = {Decorated Teichm{\"u}ller theory},
  year      = {2012},
  isbn      = {978-3-03719-075-3},
  series    = {The QGM Master Class Series},
  volume    = {1},
}

@Article{FockRosly,
  author    = {Fock, Vladimir V. and Rosly, Alexei A.},
  journal   = {Translations of the American Mathematical Society-Series 2},
  title     = {Poisson structure on moduli of flat connections on {Riemann} surfaces and r-matrix},
  year      = {1999},
  month     = feb,
  note      = {arXiv: math/9802054},
  pages     = {67--86},
  volume    = {191},
  publisher = {Providence [etc.] American Mathematical Society, 1949-},
}

@Article{AKSM,
  author    = {Alekseev, Anton and Kosmann-Schwarzbach, Yvette. and Meinrenken, Eckhard},
  journal   = {Canadian Journal of Mathematics},
  title     = {Quasi-{Poisson} {Manifolds}},
  year      = {2002},
  issn      = {0008-414X, 1496-4279},
  month     = feb,
  number    = {1},
  pages     = {3--29},
  volume    = {54},
  doi       = {10.4153/CJM-2002-001-5},
  publisher = {Cambridge University Press},
}

@Article{AKS,
  author    = {Alekseev, Anton and Kosmann-Schwarzbach, Yvette},
  journal   = {Journal of Differential geometry},
  title     = {Manin pairs and moment maps},
  year      = {2000},
  number    = {1},
  pages     = {133--165},
  volume    = {56},
  publisher = {Lehigh University},
}

@Article{AtiyahBott,
  author    = {Atiyah, Michael F. and Bott, Raoul},
  journal   = {Philosophical Transactions of the Royal Society of London. Series A, Mathematical and Physical Sciences},
  title     = {The {Yang}-{Mills} {Equations} over {Riemann} {Surfaces}},
  year      = {1983},
  issn      = {0080-4614},
  number    = {1505},
  pages     = {523--615},
  volume    = {308},
  publisher = {The Royal Society},
  url       = {https://www.jstor.org/stable/37156},
}

@inbook{LBS:symplectic-poisson,
  title = {Symplectic and Poisson Geometry of the Moduli Spaces of Flat Connections Over Quilted Surfaces},
  ISBN = {9783319099491},
  ISSN = {2352-3905},
  DOI = {10.1007/978-3-319-09949-1_11},
  booktitle = {Mathematical Aspects of Quantum Field Theories},
  publisher = {Springer International Publishing},
  author = {Li-Bland,  David and Ševera,  Pavol},
  year = {2015},
  pages = {343–411},
      archivePrefix={arXiv},
      primaryClass={math.DG},
      url={https://arxiv.org/abs/1304.0737}, 
}

@Article{LBS:poisson,
  author        = {Li-Bland, David and {\v{S}}evera, Pavol},
  journal       = {Documenta Mathematica},
  title         = {Moduli Spaces for Quilted Surfaces and Poisson Structures},
  year          = {2015},
  pages         = {1071--1135},
  volume        = {20},
  archiveprefix = {arXiv},
  eprint        = {1212.2097},
  primaryclass  = {math.DG},
}

@Article{Cartier,
  author   = {Cartier, Pierre},
  journal  = {Les rencontres physiciens-mathématiciens de Strasbourg -RCP25},
  title    = {Construction combinatoire des invariants de {Vassiliev}-{Kontsevich} des nœuds},
  year     = {1993},
  pages    = {1--10},
  volume   = {45},
  language = {fr},
  url      = {http://www.numdam.org/item/RCP25_1993__45__1_0/},
}

@Article{LBS:DQ,
  author        = {Li-Bland, David and \v{S}evera, Pavol},
  journal       = {International Mathematics Research Notices},
  title         = {On Deformation Quantization of Poisson–Lie Groups and Moduli Spaces of Flat Connections},
  year          = {2014},
  issn          = {1687-0247},
  month         = {Sep},
  number        = {15},
  pages         = {6734 - 6751},
  volume        = {2015},
  archiveprefix = {arXiv},
  doi           = {10.1093/imrn/rnu154},
  eprint        = {1307.2047},
  primaryclass  = {math.QA},
  publisher     = {Oxford University Press},
}

@Article{KontsevichVassiliev,
  author  = {Kontsevich, Maxim},
  journal = {Adv. in Sov. Math},
  title   = {Vassiliev’s knot invariants},
  year    = {1993},
  number  = {2},
  pages   = {137--150},
  volume  = {16},
}

@Article{AMM,
  author        = {Alekseev, Anton and Malkin, Anton and Meinrenken, Eckhard},
  journal       = {Journal of Differential Geometry},
  title         = {Lie group valued moment maps},
  year          = {1998},
  issn          = {0022-040X},
  month         = {Jan},
  number        = {3},
  volume        = {48},
  archiveprefix = {arXiv},
  doi           = {10.4310/jdg/1214460860},
  eprint        = {dg-ga/9707021},
  primaryclass  = {dg-ga},
  publisher     = {International Press of Boston},
}

@Article{BZBJ:integrating,
  author        = {Ben‐Zvi, David and Brochier, Adrien and Jordan, David},
  journal       = {Journal of Topology},
  title         = {Integrating quantum groups over surfaces},
  year          = {2018},
  issn          = {1753-8424},
  month         = {Aug},
  number        = {4},
  pages         = {874–917},
  volume        = {11},
  archiveprefix = {arXiv},
  doi           = {10.1112/topo.12072},
  eprint        = {1501.04652},
  primaryclass  = {math.QA},
  publisher     = {Wiley},
}

@Article{AGS95,
  author    = {Alekseev, Anton and Grosse, Harald and Schomerus, Volker},
  journal   = {Communications in Mathematical Physics},
  title     = {Combinatorial quantization of the Hamiltonian Chern-Simons theory I},
  year      = {1995},
  number    = {2},
  pages     = {317--358},
  volume    = {172},
  publisher = {Springer},
}

@Article{RocheSzenes02,
  author    = {Roche, Philippe and Szenes, András},
  journal   = {Advances in Mathematics},
  title     = {Trace Functionals on Noncommutative Deformations of Moduli Spaces of Flat Connections},
  year      = {2002},
  issn      = {0001-8708},
  month     = {Jun},
  number    = {2},
  pages     = {133–192},
  volume    = {168},
  doi       = {10.1006/aima.2001.2045},
  publisher = {Elsevier BV},
}

\end{document}